\documentclass[12pt]{amsart}

\usepackage{enumerate}

\newtheorem{theorem}{Theorem}[section]

\newtheorem{lemma}[theorem]{Lemma}

\newtheorem{proposition}[theorem]{Proposition}
\newtheorem{corollary}[theorem]{Corollary}

\theoremstyle{remark}

\newtheorem*{notation}{Notation}

\newcommand{\bC}{\mathbb{C}}
\newcommand{\bQ}{\mathbb{Q}}
\newcommand{\bR}{\mathbb{R}}
\newcommand{\bZ}{\mathbb{Z}}
\newcommand{\cont}{\mathrm{cont}}
\newcommand{\cC}{\mathcal{C}}
\newcommand{\cE}{\mathcal{E}}
\newcommand{\cO}{\mathcal{O}}
\newcommand{\den}{\mathrm{den}}
\newcommand{\et}{\quad\text{and}\quad}
\newcommand{\GL}{\mathrm{GL}}
\newcommand{\SL}{\mathrm{SL}}
\newcommand{\trace}{\mathrm{trace}}
\newcommand{\uT}{\mathbf{T}}
\newcommand{\uu}{\mathbf{u}}
\newcommand{\ux}{\mathbf{x}}
\newcommand{\uy}{\mathbf{y}}
\newcommand{\vol}{\mathrm{vol}}
\numberwithin{equation}{section}
\begin{document}

\baselineskip=15pt

\title[Measures of approximation to Markoff extremal numbers]
{Measures of algebraic approximation\\ to Markoff extremal numbers}
\author{Damien ROY}
\author{Dmitrij Zelo}
\address{
   D\'epartement de Math\'ematiques\\
   Universit\'e d'Ottawa\\
   585 King Edward\\
   Ottawa, Ontario K1N 6N5, Canada}
\email[Damien Roy]{droy@uottawa.ca}
\email[Dmitrij Zelo]{dzelo097@uottawa.ca}
\thanks{Research partially supported by NSERC}
%\date{4 November 2009}

\begin{abstract}
Let $\xi$ be a real number which is neither rational nor quadratic over $\bQ$.  Based on work of Davenport and Schmidt, Bugeaud and Laurent have shown that, for any real number $\theta$, there exist a constant $c>0$ and infinitely many non-zero polynomials $P\in\bZ[T]$ of degree at most $2$ such that $|\theta-P(\xi)|\le c \|P\|^{-\gamma}$ where $\gamma=(1+\sqrt{5})/2$ denotes for the golden ratio and where the norm $\|P\|$ of $P$ stands for the largest absolute value of its coefficients.  In the present paper, we show conversely that there exists a class of transcendental numbers $\xi$ for which the above estimates are optimal up to the value of the constant $c$ when one takes $\theta=R(\xi)$  for a polynomial $R\in\bZ[T]$ of degree $d\in\{3,4,5\}$ but curiously not for degree  $d=6$, even with $\theta=2\xi^6$.
\end{abstract}

%%%%%%%%%%%%%%%%%%%%%%%%%%%%%%%%%%%%%%%%%%%%%%%%%%%%%%%%%%%%%%%%%%%%%
%%%%%%%%%%%%%%%%%%%%%%%%%%%%%%%%%%%%%%%%%%%%%%%%%%%%%%%%%%%%%%%%%%%%%

\maketitle

\section{Introduction}

Define the \emph{norm} $\|P\|$ of a polynomial $P\in\bC[T]$ as the largest absolute value of its coefficients, and the \emph{height} $H(\alpha)$ of an algebraic number $\alpha$ as the norm of its irreducible polynomial in $\bZ[T]$.  Denote also by $\gamma=(1+\sqrt{5})/2$ the golden ratio. In their study of approximation to real numbers by algebraic integers of bounded degree, H.~Davenport and W.~M.~Schmidt established that, for each real number $\xi$ which is neither rational nor quadratic over $\bQ$, there exist infinitely many algebraic integers of degree at most $3$ satisfying
\begin{equation}
\label{intro:eq:DS}
 |\xi-\alpha| \le c H(\alpha)^{-\gamma-1}
\end{equation}
for a suitable positive constant $c$ depending only on $\xi$ \cite[Thm.~1]{DS}.  Some years ago, Y.~Bugeaud and O.~Teuli\'e observed that the same result still holds if we restrict to algebraic integers of degree exactly $3$ (see \cite{BT} and the refinement in \cite{Te}).  By another adaptation of the method of Davenport and Schmidt, Teuli\'e also noted in \cite{Te} that the same result holds with algebraic units $\alpha$ of degree $4$ and norm $1$.  He could also have used algebraic integers of degree $4$ and trace $0$ or units of degree $5$ and trace $0$.  Many variations are possible.  From a different perspective, Y.~Bugeaud and M.~Laurent proved that, for any $\theta\in\bR$, there exist a constant $c=c(\xi,\theta)>0$ and infinitely many non-zero polynomials $P\in\bZ[T]$ of degree at most $2$ such that
\begin{equation}
\label{intro:eq:BL}
 |\theta-P(\xi)| \le c \|P\|^{-\gamma}
\end{equation}
\cite[Prop.~1, Cor.~(iii)]{BL}.  When $\theta=\xi^3$, this follows from the original result of Davenport and Schmidt mentioned above.  When $\theta=\xi^3+1/\xi$, it follows from the result of Teuli\'e in \cite{Te}.

In \cite{RcubicII}, it is shown that, for a non-empty set of transcendental numbers $\xi$, the above-mentioned result of Davenport and Schmidt is optimal up to the value of the constant $c$.  By \cite[Thm.~6.2]{Rdio}, this is also true of \eqref{intro:eq:BL} with $\theta=\xi^3$.  The purpose of the present paper, in continuation to \cite[Chap.~4]{Ze}, is to extend these observations to a larger class of the above mentioned estimates.  To this end, we recall that all of these estimates derive from a single result via geometry of numbers.  The latter, again due to Davenport and Schmidt, reads as follows.  For each real number $\xi$ which is not rational nor quadratic over $\bQ$, there exist a constant $c=c(\xi)>0$ and arbitrarily large real numbers $X\ge 1$ for which the inequalities
\begin{equation}
\label{intro:eq:DSbis}
 |x_0|\le X, \quad |x_0\xi-x_1| \le cX^{-1/\gamma}, \quad |x_0\xi^2-x_2| \le cX^{-1/\gamma}
\end{equation}
have no non-zero solution $(x_0,x_1,x_2)$ in $\bZ^3$ \cite[Thm.~1a]{DS}.  Moreover, the method is such that, for a given $\xi$, any strengthening of this statement which replaces the constant $c$ by a function of $X$ tending to infinity with $X$ leads to corresponding improvements in all the above mentioned results of approximation to this number $\xi$.  In \cite{RcubicI}, it is shown that there exist countably many real numbers $\xi$ which are not rational nor quadratic over $\bQ$, for which no such strengthening is possible.  We should therefore restrict to these numbers henceforth called \emph{extremal}.

For each extremal number $\xi$, there exists, by definition, a constant $c>0$ such that the inequalities \eqref{intro:eq:DSbis} have a non-zero solution $(x_0,x_1,x_2) \in \bZ^3$ for any $X\ge 1$.  Although the extremal numbers are transcendental over $\bQ$ (this follows from Schmidt' subspace theorem), they are in this respect the transcendental numbers which behave the most closely like quadratic real numbers.  Moreover, each extremal real number $\xi$ comes with a sequence of ``best quadratic approximations'' $(\alpha_k)_{k \ge 1}$ which is uniquely determined by $\xi$ up to its first terms. In \cite[Prop.~4.9]{Rmarkoff}, it is shown that the sequence of their conjugates $(\bar{\alpha}_k)_{k\ge 1}$ admits exactly two accumulation points, called the \emph{conjugates} of $\xi$.

Let $\{\xi\}$ denote the distance from a real number $\xi$ to a closest integer.  The Lagrange constant of a real number $\xi$ is $\nu(\xi):=\liminf_{n\to\infty} n\{n\xi\}$.  When $\xi$ is not rational nor quadratic over $\bQ$, Markoff's theory tells us that $\nu(\xi)\le 1/3$ (see \cite{Ma,Mb} or \cite[Ch.~2]{Ca}).  In \cite{Rmarkoff}, it is shown that there exist extremal numbers $\xi$ with largest possible Lagrange constant $\nu(\xi)=1/3$.  We call them \emph{Markoff extremal numbers}.  By \cite[Lemma 4.3]{Rmarkoff}, a countable subset of such numbers $\xi$ have conjugates $\xi\pm 3$.  The first main result of this paper reads as follows.

\begin{theorem}
\label{intro:thmPR}
Let $\xi$ be a Markoff extremal number whose conjugates are $\xi\pm 3$, and let $d\in\{3, 4, 5\}$.  For any pair of
polynomials $P, R\in\bZ[T]$ with $\deg(R)=d$ and $\deg(P)\le 2$, we have
\[
 |P(\xi)+R(\xi)| \ge
  c (1+\|P\|)^{-\gamma}\cdot
  \begin{cases}
    \|R\|^{-\gamma^4} &\ \text{if\, $d=3$,}\\[3pt]
    \|R\|^{-\gamma^{d+7}} &\ \text{if\, $d=4$ or $d=5$,}
  \end{cases}
\]
for a constant $c=c(\xi)>0$.
\end{theorem}

For $d\in\{3,4\}$, this refines \cite[Thm.~4.2.1]{Ze}, although the latter result applies to a larger class of extremal numbers (the same as in \cite{RcubicII}).  The reader will note that we put special care in all estimates for degree $d=3$. As an immediate consequence, we get:

\begin{corollary}
\label{intro:cor1}
Let $\xi$ be as in Theorem \ref{intro:thmPR} and let $R\in\bZ[T]$ with $3\le\deg(R)\le 5$.
For any polynomial $P\in\bZ[T]$ with $\deg(P)\le 2$ and for any root $\alpha$ of the sum $P+R$,
we have $|\xi-\alpha| \ge c H(\alpha)^{-\gamma-1}$ with a constant $c=c(\xi,R)>0$.
\end{corollary}

Upon choosing $R(T)=T^3$ or $R(T)=T^4$, this gives in particular:

\begin{corollary}
\label{intro:cor2}
Let $\xi$ be as in Theorem \ref{intro:thmPR}.  For any algebraic integer $\alpha$ having degree at most $3$ or degree $4$ and trace $0$, we have $|\xi-\alpha| \ge c H(\alpha)^{-\gamma-1}$ with $c=c(\xi)>0$.
\end{corollary}

Extending \cite[Thm.~1.1]{RcubicII}, Corollary \ref{intro:cor2} shows that, for each Markoff extremal number, both the first result of Davenport and Schmidt mentioned in the introduction and the analogous result of approximation by algebraic integers of degree at most $4$ and trace $0$ are optimal up to the value of the constant $c$.  It is possible that this is also true of the result of Teuli\'e about approximation by algebraic units of degree at most $4$, but we did not consider that case.  However the next result shows that Theorem \ref{intro:thmPR} curiously does not extend to degree $d=6$.

\begin{theorem}
\label{intro:thmdeg6}
Let $\xi$ be as in Theorem \ref{intro:thmPR}.  There are infinitely many algebraic numbers $\alpha$ which are roots of polynomials of the form $2T^6+a_2T^2+a_1T+a_0$ with $a_0,a_1,a_2\in\bZ$ and satisfy
\[
 |\xi-\alpha| \le c H(\alpha)^{-\gamma-1}(\log \log H(\alpha))^{-1}
\]
for a constant $c=c(\xi)>0$.
\end{theorem}

By a cleaver application of the effective subspace theorem, B.~Adamczewski and Y.~Bugeaud recently established a measure of transcendence that applies to any extremal number \cite[Thm.~5.5]{AB}.  They showed that, given $\xi$ extremal, there exists a constant $c=c(\xi)>0$ such that, for any integer $d\ge 1$ and any non-zero polynomial $R\in\bZ[T]$ of degree at most $d$, we have
\begin{equation}
\label{intro:eq:AB}
 |R(\xi)| \ge \|R\|^{-\exp\{c(\log 3d)^2(\log \log 3d)^2\}}.
\end{equation}
For $d=1$ and $d=2$, completely explicit estimates of this sort follow respectively from Theorems 1.3 and 1.4 of \cite{RcubicI}.  In higher degree, nothing is known besides \eqref{intro:eq:AB}.  Here, by restriction to the present smaller class of extremal numbers, we obtain:

\begin{theorem}
\label{intro:thmR}
Let $\xi$ be as in Theorem \ref{intro:thmPR}. For any polynomial $R(T)\in\bZ[T]$ of degree $d\le 6$, we have
\[
 |R(\xi)| \gg
  \begin{cases}
    \|R\|^{-\gamma^3} &\ \text{if\, $d\le 3$,}\\[3pt]
    \|R\|^{\gamma^2-2\gamma^d} &\ \text{if\, $4\le d\le 6$.}
  \end{cases}
\]
\end{theorem}

There are evidences that the second estimate extends at least up to $d=9$ with the elementary method of proof that we develop in \S4.  If it holds for each $d\ge 4$, this would improve on \eqref{intro:eq:AB} for these numbers $\xi$.  However, it is likely that a much stronger estimate applies.  The question is open.

The basic strategy of proof in this paper is essentially the same as that of \cite{RcubicII}.  It is inspired by \cite[Prop.~9.1]{RcubicI} (or \cite[Prop.~2.2]{RcubicII}) and makes extensive use of the algebraic and arithmetic properties of certain sequences of integer matrices attached to Markoff extremal numbers (see the next section). Each of our main results stated above requires showing that certain sequences of positive real numbers in the interval $(0,1)$ are bounded away from $0$.  This is done by showing first that the sequences in question admit finitely many accumulation points and then that these accumulation points are non-zero.  The last step is the most delicate one.  In \cite{RcubicII} this is achieved by comparison with a weaker lower bound that applies to any extremal number (see \cite[Prop.~9.2]{RcubicI}), but this type of argument is difficult to generalize.  Here we establish that the accumulation points are non-zero, by showing basically that they admit too good rational approximations to be themselves rational numbers.  Section \ref{sec:arith} provides the key arithmetic result needed to complete this program.  It also illustrates our strategy with the proof of a qualitative version of Theorem \ref{intro:thmPR} in degree $d=3$.  Section \ref{sec:deg3} proves Theorems \ref{intro:thmPR} and \ref{intro:thmR} for $d\le 3$, while Section \ref{sec:deg4-5} does it for $d=4$ and $d=5$.  The last section \ref{sec:deg6}, inspired by \cite[Prop.~10.1]{RcubicII}, is devoted to the proof of Theorem \ref{intro:thmdeg6}.  It also completes the proof of Theorem \ref{intro:thmR} for $d=6$ and shows additional surprising properties of Markoff extremal numbers which suggest interesting avenues for further research.

%%%%%%%%%%%%%%%%%%%%%%%%%%%%%%%%%%%%%%%%%%%%%%%%%%%%%%%%%%%%%%%%%%%
%
%   Preliminaries
%
%%%%%%%%%%%%%%%%%%%%%%%%%%%%%%%%%%%%%%%%%%%%%%%%%%%%%%%%%%%%%%%%%%%

\section{Preliminaries}
\label{sec:prelim}

Define the norm $\|A\|$ of a matrix $A$ with real coefficients as the maximum of the absolute values of its coefficients.  All computations in this paper ultimately rely on the following fact.

\begin{theorem}
\label{prelim:thm}
A real number $\xi$ is a Markoff extremal number with conjugates $\xi-3$ and $\xi+3$ if and only if there exists an unbounded sequence of symmetric matrices $(\ux_k)_{k\ge 1}$ in $\SL_2(\bZ)$ such that, for each $k\ge 1$, we have
\begin{equation}
\label{prelim:thm:eq1}
 \|\ux_{k+1}\|\asymp \|\ux_k\|^\gamma, \quad
 \|(\xi,-1)\ux_k\| \ll \|\ux_k\|^{-1},
\end{equation}
with implied constants that do not depend on $k$, and
\begin{equation}
\label{prelim:thm:eq2}
 \ux_{k+2} = \ux_{k+1}M_{k+1}\ux_k = \ux_k M_k \ux_{k+1}
 \quad \text{where}\quad
 M_k=\begin{pmatrix}3&(-1)^k\\(-1)^{k+1}&0\end{pmatrix}.
\end{equation}
\end{theorem}

The reader can, if he wishes, skip the proof given below which makes extensive use of the results of \cite{Rmarkoff}.  He could then take the above characterization has a working definition of the numbers that we consider in this paper.  The reason why we proceeded otherwise are simply aesthetical, although the abstract definition may suggest other approaches to the problem.  Again note that the class of numbers considered in \cite{RcubicII} is more general than the one described by the above theorem.  As \cite[Chap.~4]{Ze} suggests, most results that we prove here probably extend to that larger class of numbers but the computations would be more involved.

\begin{proof}
We first recall that the set of Markoff extremal numbers is stable under the action of $\GL_2(\bZ)$ on $\bR\setminus\bQ$ by linear fractional transformations since both the set of extremal numbers and the set of irrational numbers $\xi$ with $\nu(\xi)=1/3$ are stable under this action.

Let $\cE_3^+$ denote the set of real numbers $\xi$ for which there exists an unbounded sequence of symmetric matrices $(\ux_i)_{i\ge 1}$ in $\SL_2(\bZ)$ satisfying \eqref{prelim:thm:eq1} and \eqref{prelim:thm:eq2} for each $i\ge 1$.  Then, according to \cite[Prop.~3.1]{Rmarkoff}, each element $\xi$ of $\cE_3^+$ is extremal and, in view of \cite[Def.~4.2 and Prop.~4.9]{Rmarkoff}, its conjugates are $\xi-3$ and $\xi+3$.  Moreover, by \cite[Thm.~3.6]{Rmarkoff}, such $\xi$ is equivalent under $\GL_2(\bZ)$ to some specific element $\xi^*$ of $\cE_3^+$ which, thanks to \cite[Cor.~5.10]{Rmarkoff}, has $\nu(\xi^*)=1/3$.  Thus, we also have $\nu(\xi)=1/3$.

Conversely, suppose that $\xi$ is a Markoff extremal number with conjugates $\xi-3$ and $\xi+3$.  By \cite[Thm.~7.7]{Rmarkoff}, the fact that $\xi$ is an extremal number with $\nu(\xi)=1/3$ implies that $\xi$ is the image $A\cdot\xi^*$ of some number $\xi^*\in\cE_3^+$ under the action of some matrix $A\in\GL_2(\bZ)$.  The conjugates of $\xi^*$ being $\xi^*-3$ and $\xi^*+3$, it follows from \cite[Cor.~4.10]{Rmarkoff} that those of $\xi$ are $A\cdot(\xi^*-3)$ and $A\cdot(\xi^*+3)$.  As we assume that the conjugates of $\xi$ are $\xi-3$ and $\xi+3$, this implies that $A$ is upper triangular and so $\xi=\pm \xi^*+b$ for some integer $b$ and some choice of sign $\pm$.  By \cite[Lemma 3.5]{Rmarkoff}, we conclude that $\xi\in\cE_3^+$.
\end{proof}

\medskip
\begin{notation}
For the rest of this paper, we fix an extremal number $\xi$ as in the statement of Theorem \ref{prelim:thm}, and a corresponding sequence of symmetric matrices $(\ux_k)_{k\ge 1}$.  We write
\[
 \begin{gathered}
 J=\begin{pmatrix} 0&1\\ -1&0\end{pmatrix},
   \quad
 P=\begin{pmatrix} 3&0\\ 0&0\end{pmatrix},
   \quad
 M_k=\begin{pmatrix} 3&(-1)^k\\ (-1)^{k+1}&0\end{pmatrix} = P+(-1)^kJ,\\
 \ux_k=\begin{pmatrix} x_{k,0}&x_{k,1}\\ x_{k,1}&x_{k,2}\end{pmatrix}
 \et
 X_k = \|\ux_k\| := \max\{|x_{k,0}|,\,|x_{k,1}|,\,|x_{k,2}|\} \quad (k\ge 1).
 \end{gathered}
\]
\end{notation}

With this notation, the estimates \eqref{prelim:thm:eq1} become
\begin{equation}
\label{prelim:eq:basics}
 X_{k+1}\asymp X_k^\gamma
 \et
 x_{k,j} = x_{k,0}\xi^j + \cO(X_k^{-1}) \quad (j=0,1,2).
\end{equation}
In the sequel, we use intensively these estimates, often without reference to \eqref{prelim:eq:basics}.  In particular they imply that
\begin{equation}
\label{prelim:hyp}
 X_k < X_{k+1}
 \et
 X_k < (1+\xi^2)|x_{k,0}|
\end{equation}
for each sufficiently large integer $k$.  By removing the first matrices of the sequence $(\ux_k)_{k\ge 1}$ if necessary, we shall assume that the inequalities \eqref{prelim:hyp} hold for each $k\ge 1$.  In connection with the first estimate of \eqref{prelim:eq:basics}, we also note that the equality $\gamma^2=\gamma+1$ implies that $X_{k+2}\asymp X_{k+1}X_k$ for each $k\ge 1$.  The next lemma gathers several consequences of the recurrence formulas \eqref{prelim:thm:eq2}.

\begin{lemma}
 \label{prelim:lemma:identities}
For each integer $k\ge 2$, we have
\begin{itemize}
 \item[(i)] $\ux_{k+2} = 3x_{k,0}\ux_{k+1} - \ux_{k-1}$,
 \smallskip
 \item[(ii)] $\ux_{k}J {\bf x}_{k+1} = J M_{k} \ux_{k-1}$,
 \smallskip
 \item[(iii)] $\ux_{k}J {\bf x}_{k+2} = J M_{k} \ux_{k+1}$,
 \smallskip
 \item[(iv)] $\ux_kJ\ux_{k+4} = JM_k\ux_{k+1}P\ux_{k+3} - (-1)^k \ux_{k+2}$.
\end{itemize}
\end{lemma}

\begin{proof}
As in the proof of \cite[Lemma 2.5]{RcubicII}, we note that, for $k\ge 2$, the recurrence relation \eqref{prelim:thm:eq2} leads to
$\ux_{k+2} = \ux_{k} M_{k} \ux_{k+1} = (\ux_{k}M_{k})^2 \ux_{k-1}$.
Applying Cayley-Hamilton's theorem to the matrix $\ux_{k}M_{k}$ with trace $3x_{k,0}$ and determinant $1$, we obtain :
\[
 \ux_{k+2} = (3x_{k,0}\ux_{k}M_{k}-I)\,\ux_{k-1} = 3x_{k,0}\ux_{k+1} -\ux_{k-1}.
\]
The next two formulas (ii) and (iii) follow directly from \eqref{prelim:thm:eq2} together with the identity $^t\ux J\ux = J$ valid for any matrix $\ux\in\SL_2(\bR)$. We find
\[
 \ux_{k}J {\bf x}_{k+1}
  = \ux_{k} J {\bf x}_{k}M_k \ux_{k-1}
  = J M_{k} \ux_{k-1}
 \et
 \ux_{k}J {\bf x}_{k+2}
  = \ux_{k} J {\bf x}_{k}M_k \ux_{k+1}
  = J M_{k} \ux_{k+1}.
\]
Finally, the last formula (iv) follows from the following computations
\[
 \begin{aligned}
 \ux_kJ\ux_{k+4}
  &= \ux_k J\ux_{k+2} M_{k+2} \ux_{k+3}
    &&\text{by \eqref{prelim:thm:eq2},}\\
  &= J M_k \ux_{k+1} (P+(-1)^kJ) \ux_{k+3}
    &&\text{by (iii),}\\
  &= J M_k \ux_{k+1}P \ux_{k+3}+ (-1)^k J M_k J M_{k+1} \ux_{k+2}
    &&\text{by (iii) again,}\\
  &= J M_k \ux_{k+1}P \ux_{k+3} - (-1)^k \ux_{k+2}.
 \end{aligned}
\]
\end{proof}

Comparing coefficients on both sides of the equality (i) of Lemma \ref{prelim:lemma:identities}, we obtain
\begin{equation}
\label{prelim:eqrecj}
x_{k+2,j}=3x_{k,0}x_{k+1,j}-x_{k-1,j} \quad (j=0,1,2).
\end{equation}
Doing the same with the three other equalities (ii) to (iv), we obtain ``commutation formulas'' which play an important role in the sequel.  Each matrix equality gives rise to four identities among which we retain only three.  Equality (ii) leads to
\begin{align}
\label{prelim:eq:1,0,1}
x_{k,0}x_{k+1,1}
  &= x_{k,1}x_{k+1,0} - (-1)^{k} x_{k-1,0},\\
\label{prelim:eq:1,0,2}
x_{k,0}x_{k+1,2}
  &= x_{k,1}x_{k+1,1} - (-1)^{k} x_{k-1,1},\\
\label{prelim:eq:1,1,2}
x_{k,1}x_{k+1,2}
  &= x_{k,2}x_{k+1,1} - 3 x_{k-1,1} - (-1)^{k} x_{k-1,2}.
\end{align}
Similarly (iii) provides
\begin{align}
\label{prelim:eq:2,0,1}
x_{k,0}x_{k+2,1}
  &= x_{k,1}x_{k+2,0} - (-1)^{k} x_{k+1,0},\\
\label{prelim:eq:2,0,2}
x_{k,0}x_{k+2,2}
  &= x_{k,1}x_{k+2,1} - (-1)^{k} x_{k+1,1},\\
\label{prelim:eq:2,1,2}
x_{k,1}x_{k+2,2}
  &= x_{k,2}x_{k+2,1} - 3 x_{k+1,1} - (-1)^{k} x_{k+1,2}.
\end{align}
Finally (iv) gives
\begin{align}
\label{prelim:eq:4,0,1}
x_{k,0}x_{k+4,1}
  &= x_{k,1}x_{k+4,0} - 3 (-1)^{k} x_{k+1,0} x_{k+3,0} - (-1)^k x_{k+2,0},\\
\label{prelim:eq:4,0,2}
x_{k,0}x_{k+4,2}
  &= x_{k,1}x_{k+4,1} - 3 (-1)^k x_{k+1,0} x_{k+3,1}- (-1)^k x_{k+2,1},\\
\label{prelim:eq:4,1,2}
x_{k,1}x_{k+4,2}
  &= x_{k,2}x_{k+4,1} - 3( 3x_{k+1,0} + (-1)^k x_{k+1,1})x_{k+3,1} -(-1)^k x_{k+2,2}.
\end{align}

We conclude this section with two series of estimates which derive from these formulas.

\begin{lemma}
\label{lemmaAF}
For each index $k\ge 2$, we have
\begin{align*}
\mathrm{(a)}\ \
x_{k,0}x_{k+2,2}\xi
   &= A_k + \cO(X_{k+1}^{-1})
   &\text{where }
   A_k &= x_{k,1}x_{k+2,2} - (-1)^{k} x_{k+1,2}\,,\\
\mathrm{(b)}\ \
x_{k,1}x_{k+2,2}\xi
   &= B_k - (-1)^{k} x_{k+1,2}\xi + \cO(X_{k+1}^{-1})
   &\text{where }
   B_k &= x_{k,2}x_{k+2,2} - 3 x_{k+1,2}\,,\\
\mathrm{(c)}\ \
x_{k,0}x_{k+1,2}\xi
   &= C_k + \cO(X_{k-1}^{-1})
   &\text{where }
   C_k &= x_{k,1}x_{k+1,2} - (-1)^{k} x_{k-1,2}\,,\\
\mathrm{(d)}\ \
x_{k,1}x_{k+1,2}\xi
   &= D_k - (-1)^{k} x_{k-1,2}\xi + \cO(X_{k-1}^{-1})
   &\text{where }
   D_k &= x_{k,2}x_{k+1,2} - 3 x_{k-1,2}\,,\\
\mathrm{(e)}\ \
x_{k,0}x_{k+4,2}\xi
   &= E_k + \cO(X_{k+2}^{-1})
   &\text{where }
   E_k&\in\bZ,\\
\mathrm{(f)}\ \
x_{k,1}x_{k+4,2}\xi
   &= F_k -(-1)^k x_{k+2,2}\xi + \cO(X_{k+2}^{-1})
   &\text{where }
   F_k&\in\bZ.
\end{align*}
\end{lemma}

\begin{proof}
All these estimates are proved on the same pattern. To prove the first one, we simply multiply both sides of \eqref{prelim:eq:2,0,2} by $\xi$ and then use \eqref{prelim:eq:basics}.  This gives
\[
 \begin{aligned}
 x_{k,0}x_{k+2,2}\xi
  &= x_{k,1}x_{k+2,1}\xi - (-1)^{k} x_{k+1,1}\xi \\
  &= x_{k,1}(x_{k+2,2}+\cO(X_{k+2}^{-1})) - (-1)^{k} (x_{k+1,2}+\cO(X_{k+1}^{-1}))\\
  &= A_k +\cO(X_{k+1}^{-1}).
 \end{aligned}
\]
In the same way, the remaining five estimates follow respectively from \eqref{prelim:eq:2,1,2}, \eqref{prelim:eq:1,0,2}, \eqref{prelim:eq:1,1,2}, \eqref{prelim:eq:4,0,2} and \eqref{prelim:eq:4,1,2}.
\end{proof}

\begin{lemma}
\label{prelim:lemmaX}
For each index $k\ge 1$, we have
\begin{itemize}
\item[(i)] $x_{k,0}x_{k+2,0}\xi^4
  = B_k - 2 (-1)^k x_{k+1,2}\xi + \cO(X_{k+1}^{-1})$,
\smallskip
\item[(ii)] $x_{k,0}x_{k+1,0}x_{k+3,0}\xi^5
  \equiv - 6 x_{k+1,2}\xi + \cO(X_{k+1}^{-1}) \mod \bZ$,
\smallskip
\item[(iii)] $x_{k,0}x_{k+1,2}x_{k+3,2}\xi
  \equiv - 2 x_{k+1,2}\xi + \cO(X_{k+1}^{-1}) \mod \bZ$.
\end{itemize}
\end{lemma}

\begin{proof}  Using Lemma \ref{lemmaAF} (a), we find
\[
x_{k,0}x_{k+2,0}\xi^4
  = x_{k,0}x_{k+2,2}\xi^2 + \cO(X_{k+1}^{-1})
  = A_k\xi + \cO(X_{k+1}^{-1}).
\]
Similarly, using Lemma \ref{lemmaAF} (b), we get
\begin{equation}
\label{eqAkxi}
A_k\xi = B_k - 2(-1)^k x_{k+1,2}\xi + \cO(X_{k+1}^{-1}).
\end{equation}
Putting these two results together, we obtain
the equality (i).  Replacing $k$ by $k+1$ in this equality
and multiplying both sides of the resulting equation by $x_{k,0}$,
we find
\begin{align*}
x_{k,0}x_{k+1,0}x_{k+3,0}\xi^4
  &= x_{k,0}B_{k+1} + 2(-1)^k  x_{k,0}x_{k+2,2}\xi + \cO(X_{k+1}^{-1})\\
  &= x_{k,0}B_{k+1} + 2(-1)^k  A_k + \cO(X_{k+1}^{-1}),
\end{align*}
where the last step uses Lemma \ref{lemmaAF} (a).  Multiplying by $\xi$
and using \eqref{eqAkxi}, this gives
\[
x_{k,0}x_{k+1,0}x_{k+3,0}\xi^5
  \equiv x_{k,0}B_{k+1}\xi - 4 x_{k+1,2}\xi  + \cO(X_{k+1}^{-1})
  \mod \bZ.
\]
Moreover, using Lemma \ref{lemmaAF} (a), we find
\[
x_{k,0}B_{k+1}\xi
 = x_{k,0}x_{k+1,2}x_{k+3,2}\xi - 3 A_k + \cO(X_{k+1}^{-1}).
\]
Thus, in order to prove (ii), it remains simply to verify (iii).
By \eqref{prelim:eq:1,0,2}, we have
\[
x_{k,0}x_{k+1,2}x_{k+3,2}\xi
 = ( x_{k,1}x_{k+1,1} - (-1)^{k} x_{k-1,1} ) x_{k+3,2}\xi
\]
By Lemma \ref{lemmaAF} (b) and (f), we also have
\begin{align*}
 x_{k+1,1}x_{k+3,2}\xi
   &= B_{k+1} + (-1)^k x_{k+2,2}\xi + \cO(X_{k+2}^{-1}),\\
 x_{k-1,1}x_{k+3,2}\xi
   &= F_{k-1} + (-1)^k x_{k+1,2}\xi + \cO(X_{k+1}^{-1}).
\end{align*}
Substituting these equalities into the previous one, we obtain
\[
x_{k,0}x_{k+1,2}x_{k+3,2}\xi
 \equiv (-1)^k x_{k,1}x_{k+2,2}\xi - x_{k+1,2}\xi + \cO(X_{k+1}^{-1})
 \mod \bZ.
\]
The relation (iii) follows from this together with Lemma \ref{lemmaAF} (b).
\end{proof}

%%%%%%%%%%%%%%%%%%%%%%%%%%%%%%%%%%%%%%%%%%%%%%%%%%%%%%%%%%%%%%%%%%%
%
%   Accumulation points
%
%%%%%%%%%%%%%%%%%%%%%%%%%%%%%%%%%%%%%%%%%%%%%%%%%%%%%%%%%%%%%%%%%%%

\section{Accumulation points}
\label{sec:accpts}

Recall that, for any real number $\eta$, we denote by $\{\eta\}$ its distance to a closest integer.  This function has the property that $|\{\eta\}-\{\eta'\}| \le \{\eta-\eta'\}$ for any $\eta,\eta'\in\bR$.  For an extremal number $\xi$ as we fixed in Section \ref{sec:prelim}, it is shown in \cite[\S4]{RcubicII} that the sequence $\big( \{x_{k,0}\xi^3\} \big)_{k\ge 1}$ has at most three accumulation points.  In this section, we extend this result by showing that, for any polynomial $R\in\bZ[T]$ of degree at most $5$, the sequence $\big( \{x_{k,0}R(\xi)\} \big)_{k\ge 1}$ has at most six accumulation points.

\begin{lemma}
\label{lemmaP}
We have
\begin{itemize}
\item[(i)] $\{(x_{k+3,0}+x_{k,0})\xi^j\} \ll X_k^{-1}$ \quad for $j=0,1,2,3$,
\smallskip
\item[(ii)] $\{(x_{k+6,0}-x_{k,0})\xi^j\} \ll X_k^{-1}$ \quad for $j=0,1,2,\dots,5$.
\end{itemize}
\end{lemma}

\begin{proof}
For $j=0,1,2$, the formulas (i) and (ii) are clear because the second estimate in \eqref{prelim:eq:basics} shows that $\{x_{k,0}\xi^j\}\ll X_k^{-1}$.  We find
\begin{equation}
\label{LemmaP:eq:xi^3}
\begin{aligned}
(x_{k+3,0}+x_{k,0})\xi^3
 &= (x_{k+3,2}+x_{k,2})\xi + \cO(X_k^{-1})
  &&\text{by \eqref{prelim:eq:basics},}\\
 &= 3 x_{k+1,0}x_{k+2,2}\xi + \cO(X_k^{-1})
  &&\text{by \eqref{prelim:eqrecj},}\\
 &= 3 C_{k+1} + \cO(X_k^{-1})
  &&\text{by Lemma \ref{lemmaAF} (c),}
\end{aligned}
\end{equation}
which proves (i) for $j=3$.  Since
\begin{equation}
\label{LemmaP:eq:xi^j}
(x_{k+6,0}-x_{k,0})\xi^j
 = (x_{k+6,0}+x_{k+3,0})\xi^j -(x_{k+3,0}+x_{k,0})\xi^j
\end{equation}
for any $j\ge 0$, it also proves (ii) for $j=3$.  By Lemma \ref{lemmaAF} (d),
we  have
\begin{equation}
\label{LemmaP:eqCk}
C_{k+1}\xi = D_{k+1} + 2(-1)^k x_{k,2}\xi + \cO(X_k^{-1})
\end{equation}
Combining this with \eqref{LemmaP:eq:xi^3} gives
\begin{equation}
\label{LemmaP:eq4}
(x_{k+3,0}+x_{k,0})\xi^4
  = 3( D_{k+1} + 2(-1)^k x_{k,2}\xi ) + \cO(X_k^{-1}).
\end{equation}
Applying \eqref{LemmaP:eq:xi^j} with $j=4$, this in turn leads to
\begin{equation}
\label{LemmaP:eq:xi^4}
\begin{aligned}
(x_{k+6,0}-x_{k,0})\xi^4
 &= 3(D_{k+4}-D_{k+1} - 2 (-1)^k (x_{k+3,2}+x_{k,2})\xi)
    + \cO(X_k^{-1}) \\
 &= 3(D_{k+4}-D_{k+1} - 6 (-1)^k C_{k+1}) + \cO(X_k^{-1}),
\end{aligned}
\end{equation}
where the second equality uses \eqref{LemmaP:eq:xi^3}. This proves (ii) for $j=4$.
Using \eqref{prelim:eqrecj} to expand $x_{k+4,2}$, we find
\begin{align*}
D_{k+4}-D_{k+1}
 &= x_{k+4,2}x_{k+5,2} - x_{k+1,2}x_{k+2,2} - 3 x_{k+3,2} + 3 x_{k,2} \\
 &= (3x_{k+2,0}x_{k+3,2}-x_{k+1,2})x_{k+5,2} - x_{k+1,2}x_{k+2,2}
    - 3(x_{k+3,2}+x_{k,2}) + 6 x_{k,2} \\
 &= 3x_{k+2,0}x_{k+3,2}x_{k+5,2} - x_{k+1,2}(x_{k+5,2}+x_{k+2,2})
    - 3(x_{k+3,2}+x_{k,2}) + 6 x_{k,2}.
\end{align*}
Thanks to Lemma \ref{prelim:lemmaX} (iii) and the formulas \eqref{LemmaP:eq:xi^3},
this gives
\begin{equation}
\label{LemmaP:eq:Dxi}
(D_{k+4}-D_{k+1})\xi
 \equiv -6x_{k+3,2}\xi  + 6 x_{k,2}\xi + \cO(X_k^{-1})
 \mod \bZ
\end{equation}
Combining this with \eqref{LemmaP:eq:xi^4} and \eqref{LemmaP:eqCk} and then using
once again the formulas \eqref{LemmaP:eq:xi^3}, we obtain finally
\[
(x_{k+6,0}-x_{k,0})\xi^5
 \equiv 3(-6x_{k+3,2}\xi  - 6 x_{k,2})\xi + \cO(X_k^{-1})
 \equiv \cO(X_k^{-1}) \mod \bZ,
\]
which proves (ii) for $j=5$.
\end{proof}

The above lemma admits the following immediate consequence.

\begin{proposition}
\label{propdelta}
Let $R$ be a polynomial of $\bZ[T]$ of degree at most $5$.  For any positive integer $k$, the limit
\[
 \delta_k(R(\xi)) := \lim_{i\to \infty} \{ x_{k+6i,0} R(\xi) \}
\]
exists and satisfies
\[
 \big| \delta_k(R(\xi)) - \{ x_{k,0} R(\xi) \} \big| \ll X_k^{-1} \|R\|.
\]
As a function of $k$, the quantity $\delta_k(R(\xi))$ is periodic with period $6$.  When $\deg(R)\le 3$, it admits also the period $3$.
\end{proposition}

\begin{proof}
Write $R=\sum_{j=0}^5 r_j T^j$.  Using Lemma \ref{lemmaP} (ii), we find, for each $k\ge 1$,
\begin{align*}
 \big| \{ x_{k+6,0} R(\xi) \} - \{ x_{k,0} R(\xi) \} \big|
  &\le \{ (x_{k+6,0}-x_{k,0}) R(\xi) \} \\
  &\le \sum_{j=0}^5 \{ (x_{k+6,0}-x_{k,0}) \xi^j \}\, |r_j| \\
  &\le c X_k^{-1}\|R\|
\end{align*}
with a constant $c>0$ depending only on $\xi$.  Therefore, for any pair of positive integers $k$ and $i$, we have
\[
 \big| \{ x_{k+6i,0} R(\xi) \} - \{ x_{k,0} R(\xi) \} \big|
  \le c \sum_{j=0}^{i-1} X_{k+6j}^{-1}\|R\|
  \ll X_k^{-1} \|R\|.
\]
This proves the first assertion of the proposition.  The fact that $\delta_k(R(\xi))$ is $6$-periodic follows from the definition.  However, when $\deg(R)\le 3$, Lemma \ref{lemmaP} (i) provides
\[
  \big| \{ x_{k+3,0} R(\xi) \} - \{ x_{k,0} R(\xi) \} \big|
   \le \{ (x_{k+3,0}+x_{k,0}) R(\xi) \}
   \ll X_k^{-1} \|R\|
\]
showing that $\delta_{k+3}(R(\xi))-\delta_k(R(\xi))$ tends to $0$ as $k\to \infty$.  Thus, in that case, $\delta_k(R(\xi))$ is $3$-periodic.
\end{proof}

%\newpage
%%%%%%%%%%%%%%%%%%%%%%%%%%%%%%%%%%%%%%%%%%%%%%%%%%%%%%%%%%
%
%  Arithmetic Estimates
%
%%%%%%%%%%%%%%%%%%%%%%%%%%%%%%%%%%%%%%%%%%%%%%%%%%%%%%

\section{Arithmetic estimates}
\label{sec:arith}

For the applications that we have in view in the next sections, we need to estimate the greatest common divisors of $x_{k,0}$ with each of the integers $A_k$ and $E_k$ defined in Lemma \ref{lemmaAF}.  Here, we show that both of them are equal to $1$ or $2$.  The proof relies on the properties of a sequence of polynomials which we first define below following \cite[\S 8]{RcubicI}.  We end the section by showing that, for any degree three polynomial $R\in\bZ[T]$, the accumulation points of the sequence $(\{x_{k,0}R(\xi)\})_{k\ge 1}$ are transcendental over $\bQ$.  We deduce from this a proof of a qualitative version of Theorem \ref{intro:thmPR} in degree $d=3$ which illustrates our general approach to the questions dealt with in this paper.

Put $\uT=(1,T,T^2)$ where $T$ denotes an indeterminate.  Following \cite[\S 8]{RcubicI}, we define a quadratic polynomial
\[
 Q_k(T) := \det(\uT,\ux_k,\ux_{k+1})
   = \left|\begin{matrix}
      1 &T &T^2\\ x_{k,0} &x_{k,1} &x_{k,2}\\ x_{k+1,0} &x_{k+1,1} &x_{k+1,2}
     \end{matrix}\right|
   \in \bZ[T],
\]
for each integer $k\ge 1$.  The next lemma collects most of the results that we need about these polynomials.

\begin{lemma}
\label{lemmaQ}
For each index $k\ge 2$, we have
\begin{equation}
\label{lemmaQ:eq1}
 \|Q_k\| \asymp |Q_k'(\xi)| \asymp X_{k-1}, \quad
 |Q_k(\xi)| \ll X_{k+2}^{-1} \asymp \|Q_k\|^{-\gamma^3}
\end{equation}
and the leading coefficient of $Q_k$ is $(-1)^{k-1}x_{k-1,0}$.  Moreover, the polynomials $Q_{k-1}$, $Q_k$ and $Q_{k+1}$ are linearly independent and satisfy
\begin{equation}
\label{lemmaQ:eq2}
 x_{k-1,0}Q_k(T) - x_{k,0}Q_{k+1}(T) + x_{k+1,0}Q_{k-1}(T)
 = \det(\ux_{k-1},\ux_k,\ux_{k+1})
 = -2(-1)^k.
\end{equation}
\end{lemma}

The last formula will play a crucial role in our proof of Theorem \ref{intro:thmdeg6} (see \S\ref{sec:deg6}).

\begin{proof}
The estimates \eqref{lemmaQ:eq1} follow from \cite[Prop.~8.1]{RcubicI}, and \eqref{prelim:eq:1,0,1} shows that the leading coefficient of $Q_k$ is $(-1)^{k-1}x_{k-1,0}$.  To prove \eqref{lemmaQ:eq2}, we first observe that, by virtue of the recurrence relation of Lemma \ref{prelim:lemma:identities} (i), we have
\begin{equation}
 \label{lemmaQ:eq3}
 Q_{k+1}(T)
 = \det(\uT,\ux_{k+1},3x_{k,0}\ux_{k+1}-\ux_{k-1})
 = \det(\uT,\ux_{k-1},\ux_{k+1}).
\end{equation}
We also recall that, over any commutative ring $R$, any four points $\uy_1,\dots,\uy_4$ in $R^3$ satisfy the generic linear relation
\[
 \sum_{i=1}^4 (-1)^i \det(\uy_1,\dots,\widehat{\uy_i},\dots,\uy_4)\uy_i = 0.
\]
Applying this formula to the points $\ux_{k-1}$, $\ux_k$, $\ux_{k+1}$ and $\uT$, and projecting on the first coordinate, we obtain
\[
 \begin{aligned}
 \det(\ux_{k-1}, \ux_k, \ux_{k+1})
 &= \det(\uT, \ux_k, \ux_{k+1}) x_{k-1,0}
   - \det(\uT, \ux_{k-1}, \ux_{k+1}) x_{k,0}\\
 &\hspace{50pt}+ \det(\uT, \ux_{k-1}, \ux_{k}) x_{k+1,0}\\
 &= x_{k-1,0} Q_k(T)
   - x_{k,0} Q_{k+1}(T)
   + x_{k+1,0} Q_{k-1}(T),
 \end{aligned}
\]
where the second equality uses \eqref{lemmaQ:eq3}.  To complete the proof of \eqref{lemmaQ:eq2}, we note that a direct computation using the formula (2.1) of \cite{RcubicI} gives
\[
 \begin{aligned}
 \det(\ux_{k-1}, \ux_k, \ux_{k+1})
 &= \trace(J\ux_{k-1}J\ux_{k+1}J\ux_k)\\
 &= \trace(J\ux_{k-1}J\ux_{k-1}M_{k-1}\ux_kJ\ux_k)\\
 &= -\trace(M_{k-1}J)
 = -2(-1)^k.
 \end{aligned}
\]
Finally, the defining formulas for $Q_{k-1}$ and $Q_k$ combined with the alternative formula \eqref{lemmaQ:eq3} for $Q_{k+1}$ imply that the determinant of the matrix formed by the coefficients of these three polynomials (in some order) is $\pm \det(\ux_{k-1}, \ux_k, \ux_{k+1})^2 = \pm 4$.  So these polynomials are linearly independent.
\end{proof}

We can now state and prove the main result of this section where, for a non-zero polynomial $Q\in\bZ[T]$, the notation $\cont(Q)$ stands for the \emph{content} of $Q$, namely the greatest common divisor (\emph{gcd}) of its coefficients.

\begin{proposition}
\label{propGCD}
For any sufficiently large integer $k$, we have
\[
 \gcd(x_{k,0},A_k)
 = \gcd(x_{k,0},E_k)
 = \cont(Q_{k+1}) \in \{1,2\}.
\]
\end{proposition}

\begin{proof} For each $k\ge 1$, we denote the coefficients of $Q_k$ as
\[
 Q_k(T) = a_kT^2 + b_kT + c_k.
\]
Since Lemma \ref{lemmaQ} gives $|Q_{k+1}(\xi)| \ll X_{k+3}^{-1}$ and $\|Q_{k+1}\| \ll X_k$, we deduce that
\begin{align*}
 a_{k+1} x_{k+2,2}\xi + b_{k+1} x_{k+2,2} + c_{k+1} x_{k+2,1}
  = x_{k+2,0}\xi Q_{k+1}(\xi) + \cO(X_{k+2}^{-1}\|Q_{k+1}\|)
  = \cO(X_{k+1}^{-1}).
\end{align*}
As $a_{k+1}=(-1)^k x_{k,0}$, this result combined with
Lemma \ref{lemmaAF} (a) leads to
\[
 (-1)^k A_k + b_{k+1} x_{k+2,2} + c_{k+1} x_{k+2,1}
  = \cO(X_{k+1}^{-1}).
\]
The left hand side of this equality being an integer,
it must vanish for each $k$ sufficiently large and, for
those $k$, we get
\begin{equation}
\label{propGCD:eq1}
  A_k = (-1)^{k+1} (b_{k+1} x_{k+2,2} + c_{k+1} x_{k+2,1}).
\end{equation}
Moreover the definition of $Q_{k+1}$ implies that, for any
$\uu=(u_0,u_1,u_2)\in\bR^3$, we have
\begin{equation}
\label{propGCD:eq2}
 \det(\uu,\ux_{k+1},\ux_{k+2})
 = a_{k+1} u_2 + b_{k+1} u_1 + c_{k+1} u_0.
\end{equation}
For the choice of $\uu=\ux_{k+2}$, this determinant
is $0$ and therefore, using again the fact that $a_{k+1} =
(-1)^k x_{k,0}$, we also get
\begin{equation}
\label{propGCD:eq3}
 x_{k,0} x_{k+2,2}
  = (-1)^{k+1} (b_{k+1} x_{k+2,1} + c_{k+1} x_{k+2,0}).
\end{equation}
Since $\det(\ux_{k+2})=1$, the formulas \eqref{propGCD:eq1}
and \eqref{propGCD:eq3} lead to
\begin{align*}
 \gcd(x_{k,0},A_k)
 &= \gcd(x_{k,0},\, b_{k+1} x_{k+2,1} + c_{k+1} x_{k+2,0},\,
        b_{k+1} x_{k+2,2} + c_{k+1} x_{k+2,1})\\
 &= \gcd(x_{k,0}, b_{k+1}, c_{k+1})\\
 &=\cont(Q_{k+1}).
\end{align*}
To relate this quantity to $\gcd(x_{k,0},E_k)$, we
multiply the equality (e) of Lemma \ref{lemmaAF} by
$x_{k,1}$ and subtract from this product the equality (f)
multiplied by $x_{k,0}$.  This gives
\[
 x_{k,1}E_k - x_{k,0}F_k + (-1)^k x_{k,0}x_{k+2,2}\xi
 = \cO(X_{k+1}^{-1}).
\]
Applying Lemma \ref{lemmaAF} (a), we deduce that
\[
 x_{k,1}E_k - x_{k,0}F_k = (-1)^{k+1} A_k
\]
for each sufficiently large $k$.  For those $k$, we obtain
\[
 \gcd(x_{k,0},A_k)
 = \gcd(x_{k,0},x_{k,1}E_k)
 = \gcd(x_{k,0},E_k)
\]
since the formula $\det(\ux_k)=1$ implies that $x_{k,0}$ and
$x_{k,1}$ are relatively prime.  Finally, the formula
\eqref{propGCD:eq2} shows that the content of $Q_{k+1}$
divides the integer $\det(\ux_k,\ux_{k+1},\ux_{k+2})$ which
by Lemma \ref{lemmaQ} is $2(-1)^k$. Thus that content is either $1$ or $2$.
\end{proof}

Combining the above proposition with results of the preceding sections, we deduce:

\begin{proposition}
 \label{prop:delta_xi3}
Let $\ell\in\{1,2,3\}$.  For any positive integer $k$ with $k\not\equiv \ell \mod 3$, there exists an integer $y_k$ with $\gcd(x_{k,0},y_k)\,|\,2$ and
\[
 |x_{k,0}\delta_\ell(\xi^3)-y_k|
 \asymp
 \begin{cases}
   X_{k+1}^{-1} &\text{if $k\equiv \ell+1 \mod 3$,}\\
   X_{k+2}^{-1} &\text{if $k\equiv \ell+2 \mod 3$.}
 \end{cases}
\]
If $k$ is sufficiently large, then $y_k/x_{k,0}$ is a convergent of $\delta_\ell(\xi^3)$ with denominator equal to $|x_{k,0}|$ or $|x_{k,0}|/2$.  Any other convergent of $\delta_\ell(\xi^3)$ in reduced form $p/q$ satisfies $|q\delta_\ell(\xi^3)-p|\asymp q^{-1}$.
\end{proposition}

\begin{proof}
For an integer $k$ with $k\equiv \ell+1 \mod 3$, Proposition \ref{propdelta} gives
\[
 \delta_\ell(\xi^3)
  \equiv x_{k+2,0}\xi^3 + \cO(X_{k+2}^{-1})
  \equiv x_{k+2,2}\xi + \cO(X_{k+2}^{-1}) \mod \bZ.
\]
Multiplying this congruence by $x_{k,0}$ and applying Lemma \ref{lemmaAF} (a), we deduce that
\[
 x_{k,0}\delta_\ell(\xi^3)
  \equiv A_k + \cO(X_{k+1}^{-1}) \mod x_{k,0}\bZ.
\]
Thus, there exists an integer $y_k$ with $y_k\equiv A_k \mod x_{k,0}$ such that $|x_{k,0}\delta_\ell(\xi^3)-y_k| \ll X_{k+1}^{-1}$, and, by Proposition \ref{propGCD}, we have $\gcd(x_{k,0},y_k)\,|\,2$ if $k$ is sufficiently large.

Similarly, if $k\equiv \ell+2 \mod 3$, Proposition \ref{propdelta} leads to
$\delta_\ell(\xi^3) \equiv x_{k+4,2}\xi + \cO(X_{k+4}^{-1})\mod \bZ$.
Multiplying this congruence by $x_{k,0}$ and applying Lemma \ref{lemmaAF} (e), we deduce that
\[
 x_{k,0}\delta_\ell(\xi^3)
  \equiv E_k + \cO(X_{k+2}^{-1}) \mod x_{k,0}\bZ.
\]
Thus, there exists an integer $y_k$ with $y_k\equiv E_k \mod x_{k,0}$ such that $|x_{k,0}\delta_\ell(\xi^3)-y_k| \ll X_{k+2}^{-1}$, and, by Proposition \ref{propGCD}, we again have $\gcd(x_{k,0},y_k)\,|\, 2$ if $k$ is sufficiently large.

This means that, for each sufficiently large integer $k$ with $k\not\equiv\ell \mod 3$, the ratio $y_k/x_{k,0}$ is a convergent of $\delta_\ell(\xi^3)$ with denominator $|x_{k,0}|$ or $|x_{k,0}|/2$.  Moreover, let $(p_k/q_k)_{k\ge 1}$ denote the sequence of these convergents written in reduced form and listed by increasing order of denominator.  Then the above considerations also imply that $|q_k\delta_\ell(\xi^3)-p_k| \ll q_{k+1}^{-1}$.  To complete the proof, it remains to show more precisely that $|q_k\delta_\ell(\xi^3)-p_k|\asymp q_{k+1}^{-1}$ for any $k\ge 1$, and that $|q\delta_\ell(\xi^3)-p|\asymp q^{-1}$ for any other convergent $p/q$ of $\delta_\ell(\xi^3)$ (in reduced form).

To this end, recall that if $p/q$ and $p'/q'$ are consecutive convergents of a real number $\delta$, in that order, then $|q\delta-p|\asymp (q')^{-1}$ with implied constants depending only on $\delta$.  For $\delta=\delta_\ell(\xi^3)$ and $p/q=p_k/q_k$ with $k$ large enough, we have $q'\le q_{k+1}$ and this gives
\[
 |q_k\delta_\ell(\xi^3)-p_k|\asymp (q')^{-1} \gg q_{k+1}^{-1}.
\]
Since $|q_k\delta_\ell(\xi^3)-p_k|\ll q_{k+1}^{-1}$, we conclude that $|q_k\delta_\ell(\xi^3)-p_k|\asymp q_{k+1}^{-1}$ and that $q'\asymp q_{k+1}$.  In particular, any convergent $p/q$ of $\delta_\ell(\xi^3)$ with $q_k<q<q_{k+1}$ has $q\asymp q_{k+1}$ and the next convergent $p'/q'$ also has $q'\asymp q_{k+1}$, proving that $|q\delta_\ell(\xi^3)-p|\asymp q^{-1}$.
\end{proof}

\begin{corollary}
 \label{prop:delta_R(xi):cor1}
Let $\ell\in\{1,2,3\}$ and let $R\in \bZ[T]$ with $\deg(R)=3$.  Then $\delta_\ell(R(\xi))$ is transcendental over $\bQ$.  Moreover there exists positive constants $c_1$, $c_2$ depending only on $R$ and $\xi$ such that the inequality $|\delta_\ell(R(\xi))-\alpha|\le c_1 H(\alpha)^{-\gamma-2}$ has infinitely many solutions $\alpha\in \bQ$ while $|\delta_\ell(R(\xi))-\alpha|\le c_2 H(\alpha)^{-\gamma-2}$ has no such solution.
\end{corollary}

\begin{proof}
Since $\delta_\ell(\xi^j)=0$ for $j=0,1,2$, we have $\delta_\ell(R(\xi)) \equiv r\delta_\ell(\xi^3) \mod \bZ$ where $r$ denotes the leading coefficient of $R$.  Thus, there exists an integer $a$ such that $\delta_\ell(R(\xi)) = r\delta_\ell(\xi^3)+a$, and so it suffices to prove the corollary for $R(T)=T^3$. In the notation of the proposition, for each sufficiently large integer $k$ with $k\equiv \ell+2 \mod 3$, the rational number $\alpha_k = y_k/x_{k,0}$ satisfies $|\delta_\ell(\xi^3)-\alpha_k|\asymp X_k^{-\gamma-2} \asymp H(\alpha_k)^{-\gamma-2}$ while any other convergent $\alpha$ of $\delta_\ell(\xi^3)$ has $|\delta_\ell(\xi^3)-\alpha|\gg H(\alpha)^{-\gamma-1}$.  This proves the second assertion (with $R(T)=T^3$).  The first assertion follows from this thanks to Roth's Theorem \cite[Ch.~VI, Thm.~1]{Ca}.
\end{proof}

\begin{corollary}
 \label{prop:delta_R(xi):cor2}
Let $R\in \bZ[T]$ with $\deg(R)=3$.  For any polynomial $P\in\bZ[T]$ with $\deg(P)\le 2$, we have
$|R(\xi)+P(\xi)| \ge c(\|P\|+1)^{-\gamma}$ for some $c=c(R,\xi)>0$.
\end{corollary}

\begin{proof}
On the one hand, Corollary \ref{prop:delta_R(xi):cor1} shows that the accumulation points $\delta_\ell(R(\xi))$ of the sequence $(\{x_{k,0}R(\xi)\})_{k\ge 1}$ are all non-zero.  So, there exists a constant $c_1>0$ such that $\{x_{k,0}R(\xi)\}\ge c_1$ for each $k\ge 1$.  On the other hand, we have $\{x_{k,0}\xi^j\} \ll X_k^{-1}$ for $j=0,1,2$.  Thus there exists a constant $c_2>0$ such that $\{x_{k,0}P(\xi)\}\le c_2\|P\|X_k^{-1}$ for each $P\in\bZ[T]$ with $\deg(P)\le 2$ and each $k\ge 1$.  For such a polynomial $P$, we find
\[
 X_k |R(\xi)+P(\xi)|
 \ge \{x_{k,0}(R(\xi)+P(\xi))\}
 \ge \{x_{k,0}R(\xi)\} - \{x_{k,0}P(\xi)\}
 \ge c_1 - c_2\|P\|X_k^{-1}.
\]
Then, upon choosing $k$ to be the smallest positive integer with $\|P\| \le c_1(2c_2)^{-1}X_k$, we obtain $X_k \ll (\|P\|+1)^\gamma$ and the above estimate gives $|R(\xi)+P(\xi)| \ge (c_1/2) X_k^{-1} \gg (\|P\|+1)^{-\gamma}$.
\end{proof}

%%%%%%%%%%%%%%%%%%%%%%%%%%%%%%%%%%%%%%%%%%%%%%%%%%%%%%%%%%%%%%%%%%
%
%   Degree 3
%
%%%%%%%%%%%%%%%%%%%%%%%%%%%%%%%%%%%%%%%%%%%%%%%%%%%%%%%%%%%%%%%%%%

%\newpage
\section{Degree 3}
\label{sec:deg3}

In this section, we prove Theorems \ref{intro:thmPR} and \ref{intro:thmdeg6} for polynomials of degree at most $3$.  As mentioned in Section \ref{sec:prelim}, we assume that \eqref{prelim:hyp} holds for each $k\ge 1$ (in particular the sequence $(X_k)_{k\ge 1}$ is strictly increasing).  For simplicity, by omitting the first terms of the sequence $(\ux_k)_{k\ge 1}$ and shifting indices if necessary, we will also assume that the conclusion of Proposition \ref{propGCD} holds for each $k\ge 1$.  We start with a lemma.

\begin{lemma}
\label{lemmaxR}
For any integer $k\ge 1$ and any polynomial $R\in\bZ[T]$ of
degree $3$ such that the leading coefficient of $2R$ is not
divisible by $x_{k,0}$, we have
\begin{align}
\label{lemmaxR:eq1}
 \{ x_{k+2,0}R(\xi) \}
  &\ge X_k^{-1} ( 1 - c_1 X_{k+1}^{-1}\|R\| ) \\
\label{lemmaxR:eq2}
 \{ x_{k+4,0}R(\xi) \}
  &\ge X_k^{-1} ( 1 - c_1 X_{k+2}^{-1}\|R\| )
 \end{align}
for a constant $c_1>0$ depending only on $\xi$.
\end{lemma}

\begin{proof}
Let $R(T)=r_3T^3+r_2T^2+r_1T+r_0$ be any element of $\bZ[T]$
of degree $3$ and let $k$ be any positive integer for which
$x_{k,0}$ does not divide $2r_3$.  By Lemma \ref{lemmaAF}
(a) and (e), we have
\begin{align*}
 x_{k,0}x_{k+2,0}R(\xi)
  &= r_3A_k + x_{k,0}p_{k+2} + \cO(X_{k+1}^{-1}\|R\|) \\
 x_{k,0}x_{k+4,0}R(\xi)
  &= r_3E_k + x_{k,0}p_{k+4} + \cO(X_{k+2}^{-1}\|R\|)
\end{align*}
where $p_j := r_2x_{j,2}+r_1x_{j,1}+r_0x_{j,0} \in \bZ$
for any $j\ge 1$. Dividing both sides of the above
equalities by $x_{k,0}$, we deduce that
\[
 \{ x_{k+2,0}R(\xi) \}
 \ge \Big\{ \frac{r_3A_k}{x_{k,0}} \Big\}
     - c_1 \frac{\|R\|}{X_kX_{k+1}}
 \et
 \{ x_{k+4,0}R(\xi) \}
 \ge \Big\{ \frac{r_3E_k}{x_{k,0}} \Big\}
     - c_1 \frac{\|R\|}{X_kX_{k+2}}
\]
for some constant $c_1>0$.  Since $x_{k,0}$ does not divide
$2r_3$, Proposition \ref{propGCD} implies that the ratios
$(r_3A_k)/x_{k,0}$ and $(r_3E_k)/x_{k,0}$ are not integers.
Therefore, their distance
to a closest integer is $\ge |x_{k,0}|^{-1} \ge X_k^{-1}$.
The conclusion follows.
\end{proof}

We now prove the following refinement of Theorem \ref{intro:thmR}
in degree at most $3$.

\begin{theorem}
\label{thmR3}
Any non-zero polynomial $R\in\bZ[T]$ of degree at most $3$
satisfies $|R(\xi)| \gg \|R\|^{-\gamma^3}$. When $R$ is not
divisible within $\bQ[T]$ by any polynomial of the sequence
$(Q_k)_{k\ge 1}$, we have the stronger estimate $|R(\xi)|
\gg \|R\|^{-1-\gamma^2}$.
\end{theorem}

\begin{proof}
Let $R(T)$ be an element of $\bZ[T]$ of degree $\le 3$ not
divisible within $\bQ[T]$ by any $Q_k$, let $r$ denote its coefficient of $T^3$
and let $k$ be a positive integer. There are two possibilities.

\smallskip
\textbf{Case 1:} If $x_{k,0}$ does not divide $2r$, then Lemma
\ref{lemmaxR} gives
\begin{equation}
\label{thmR3:eq2}
 X_k^{-1} ( 1 - c_1 X_{k+1}^{-1}\|R\| )
  \le \{ x_{k+2,0}R(\xi) \}
  \le X_{k+2} |R(\xi)|.
\end{equation}

\smallskip
\textbf{Case 2:} If $x_{k,0}$ divides $2r$, then Lemma \ref{lemmaQ} shows that
\[
 P(T) := 2R(T)-(-1)^k(2r/x_{k,0})TQ_{k+1}(T)
\]
is a polynomial of $\bZ[T]$ of degree $\le 2$ with
\begin{equation}
\label{thmR3:eq3}
 \|P\| \ll \|R\|
 \et
 |P(\xi)| \ll |R(\xi)|+X_k^{-1}X_{k+3}^{-1}\|R\|.
\end{equation}
Since $Q_{k+1}$ does not divide $R$, this polynomial $P$ is not a
rational multiple of $Q_{k+1}$.  As $Q_k$, $Q_{k+1}$ and $Q_{k+2}$
are linearly independent, this means that $P$, $Q_j$ and $Q_{j+1}$
are linearly independent for at least one choice of $j\in\{k,k+1\}$.
For such $j$, the matrix $(P,Q_j,Q_{j+1})$ whose rows
consist of the coefficients of $P$, $Q_j$ and $Q_{j+1}$ has a
non-zero integer determinant.  Applying Lemma 4 of \cite{AR}
and then the estimates \eqref{lemmaQ:eq1}, we obtain
\begin{align*}
 1
 &\le |\det(P,Q_j,Q_{j+1})| \\
 &\le 2\|P\|\, \|Q_{j}\|\, |Q_{j+1}(\xi)|
    + 2\|P\|\, \|Q_{j+1}\|\, |Q_j(\xi)|
    + 2\|Q_{j}\|\, \|Q_{j+1}\|\, |P(\xi)|\\
 &\ll X_{k+1}^{-1} \|P\| + X_{k+2} |P(\xi)|.
\end{align*}
Finally, combining this with \eqref{thmR3:eq3}, we conclude that
\begin{equation}
\label{thmR3:eq4}
  1 \le c_2(X_{k+1}^{-1} \|R\| + X_{k+2} |R(\xi)|)
\end{equation}
for some constant $c_2=c_2(\xi)>0$.

The above discussion shows that, for each $k\ge 1$, at least one
of the inequalities \eqref{thmR3:eq2} or \eqref{thmR3:eq4}
holds.  Upon choosing $k$ so that
\[
 X_k \le c\,\|R\| < X_{k+1},
\]
where $c=\max\{X_1, 2c_1, 2c_2\}$, this leads to
\[
 |R(\xi)|
 \ge (1/2)\min\{(X_kX_{k+2})^{-1},(c_2X_{k+2})^{-1}\}
 \gg X_k^{-1-\gamma^2}
 \gg \|R\|^{-1-\gamma^2}.
\]
This proves the second assertion of the theorem.

To complete the proof, it remains to establish the first assertion
of the theorem for the non-zero polynomials $R\in \bZ[T]$ of
degree $\le 3$ which are divisible by $Q_k$ for some index $k\ge 1$.
Since each $Q_k$ has content $1$ or $2$, such a polynomial takes
the form $R=(1/2)Q_kL$ where $L$ is a non-zero
polynomial of $\bZ[T]$ of degree at most $1$.
By \cite[Thm.~1.3]{RcubicI}, we have
\[
 |L(\xi)| \gg \|L\|^{-1} (1+\log \|L\|)^{-t}
\]
for some fixed real number $t\ge 0$. Since $\|R\|\asymp
\|Q_k\|\,\|L\|$, we conclude that
\[
 |R(\xi)|
  = |Q_k(\xi)|\,|L(\xi)|
  \gg \|Q_k\|^{-\gamma^3}\|L\|^{-1} (1+\log \|L\|)^{-t}
  \gg \|R\|^{-\gamma^3}.
\]
\end{proof}

We know by Corollary \ref{prop:delta_R(xi):cor1} that, for a degree three polynomial $R\in\bZ[T]$, the accumulation points of the sequence $(\{x_{k,0}R(\xi)\})_{k\ge 1}$ are non-zero because they are transcendental over $\bQ$.  The next lemma provides a quantitative version of this statement.  It is the key step in the proof of Theorem \ref{intro:thmPR} for degree $d=3$ which concludes this section.

\begin{lemma}
\label{lemmaxRbis}
There exist positive constants $c_2$ and $c_3$ depending only on $\xi$ with the following property. For any polynomial $R\in\bZ[T]$ of degree $3$ and any integer $k\ge 1$ with $X_k > c_2\|R\|^{\gamma^2}$, we have
$\{ x_{k,0}R(\xi) \}  \ge c_3 \|R\|^{-\gamma^2}$.
\end{lemma}

\begin{proof}
Let $R$ be a polynomial of $\bZ[T]$ of degree $3$, and let $k$, $\ell$ be positive integers.  If $|x_{\ell,0}|>2\|R\|$, $k\ge \ell+2$ and $k\equiv \ell+2 \mod 3$, the leading coefficient of $2R(T)$ is not divisible by $x_{\ell,0}$ and so the inequality \eqref{lemmaxR:eq1} of Lemma \ref{lemmaxR} combined with Proposition \ref{propdelta} gives
\[
 \{ x_{k,0}R(\xi) \}
  = \{ x_{\ell+2,0}R(\xi) \} + \cO(X_{\ell+2}^{-1}\|R\|)
  \ge X_\ell^{-1} (1-c_1X_{\ell+1}^{-1}\|R\|) + \cO(X_{\ell+2}^{-1}\|R\|).
\]
Similarly, if $|x_{\ell,0}|> 2\|R\|$, $k\ge \ell+4$ and $k\equiv \ell+4 \mod 3$, the inequality \eqref{lemmaxR:eq2} of Lemma \ref{lemmaxR} combined with Proposition \ref{propdelta} leads to
\[
 \{ x_{k,0}R(\xi) \}
  = \{ x_{\ell+4,0}R(\xi) \} + \cO(X_{\ell+4}^{-1}\|R\|)
  \ge X_\ell^{-1} (1-c_1X_{\ell+2}^{-1}\|R\|) + \cO(X_{\ell+4}^{-1}\|R\|).
\]
In both cases, we conclude that
\begin{equation}
\label{lemmaxRbis:eq1}
 \{ x_{k,0}R(\xi) \}
  \ge X_\ell^{-1} (1-cX_{\ell+1}^{-1}\|R\|)
\end{equation}
for some constant $c>0$ depending only on $\xi$.

\smallskip
Define $c' = 2\max\{X_1,c, 1+\xi^2\}$.  Since $c'\ge X_1$, there exists an integer $m\ge 2$ for which
\[
 X_{m-1} \le c' \|R\| < X_m.
\]
For this choice of $m$, we obtain $\|R\|^{\gamma^2} \ge c_2^{-1} X_{m+1}$ with a constant $c_2>0$ depending only on $\xi$.  Now, assume that $X_k > c_2\|R\|^{\gamma^2}$.  In combination with the preceding inequality, this gives $X_k>X_{m+1}$ and so $k\ge m+2$.  We set
\[
 \ell := \begin{cases}
  k-2 &\text{if $k=m+2$ or $k=m+3$,}\\
  m   &\text{if $k\ge m+4$ and $k\not\equiv m$ mod $3$,}\\
  m+1 &\text{if $k\ge m+4$ and $k\equiv m$ mod $3$.}
  \end{cases}
\]
Then, in all cases, either we have $k\ge \ell+2$ and $k\equiv \ell+2\mod 3$ or we have $k\ge \ell+4$ and $k\equiv \ell+4\mod 3$.  Moreover, as $\ell\ge m$, we get $|x_{\ell,0}| \ge 2X_\ell/c' \ge 2X_m/c' >2\|R\|$ and thus \eqref{lemmaxRbis:eq1} applies.  Finally, since $cX_{\ell+1}^{-1}\|R\| \le cX_m^{-1}\|R\| \le c/c' \le 1/2$ and since $\ell\le m+1$, the latter estimate gives
\[
 \{ x_{k,0}R(\xi) \} \ge (2X_\ell)^{-1} \ge (2X_{m+1})^{-1} \gg \|R\|^{-\gamma^2}.
\]
\end{proof}

\begin{theorem}
\label{thmPR3}
For any polynomial $R\in\bZ[T]$ of degree $3$ and any polynomial $P\in\bZ[T]$ of degree at most $2$, we have
$|R(\xi)+P(\xi)| \gg (1+\|P\|)^{-\gamma}\|R\|^{\gamma^4}$.
\end{theorem}

\begin{proof}  Fix a pair of polynomials $P$ and $R$ as in the statement of the theorem.  For any integer $k\ge 1$, we have
\begin{equation}
\label{thmPR3:eq1}
 \{ x_{k,0} R(\xi) \}
 \le \{ x_{k,0} (R(\xi)+P(\xi)) \} + \{ x_{k,0} P(\xi) \}
 \le X_k |R(\xi)+P(\xi)| + c X_k^{-1} \|P\|
\end{equation}
with a constant $c>0$ depending only on $\xi$.  Define $c'=\max\{X_1,c_2,2c/c_3\}$ where $c_2$ and $c_3$ are as in Lemma \ref{lemmaxRbis}.  Since $c'\ge X_1$, there exists an integer $k\ge 2$ such that
\[
 X_{k-1} \le c' (1+\|P\|) \|R\|^{\gamma^2} < X_k.
\]
For this choice of $k$, we have $X_k \ge c_2 \|R\|^{\gamma^2}$, and so Lemma \ref{lemmaxRbis} gives $\{ x_{k,0} R(\xi) \} \ge c_3\|R\|^{-\gamma^2}$.  We also find that $cX_k^{-1}\|P\| \le (c/c')\|R\|^{-\gamma^2} \le (c_3/2)\|R\|^{-\gamma^2}$.  Combining the last two estimates with \eqref{thmPR3:eq1}, we obtain
\[
 X_k |R(\xi)+P(\xi)| \ge (c_3/2) \|R\|^{-\gamma^2}.
\]
The conclusion follows since $X_k \ll (1+\|P\|)^\gamma\|R\|^{\gamma^3}$.
\end{proof}

%%%%%%%%%%%%%%%%%%%%%%%%%%%%%%%%%%%%%%%%%%%%%%%%%%%%%%%%%%%%%%%%%%
%
%   Degrees 4 and 5
%
%%%%%%%%%%%%%%%%%%%%%%%%%%%%%%%%%%%%%%%%%%%%%%%%%%%%%%%%%%%%%%%%%%

%\newpage
\section{Degrees 4 and 5}
\label{sec:deg4-5}

We now turn to the proof of Theorems \ref{intro:thmPR} and \ref{intro:thmR} for $d=4$ and $d=5$.  The structure of this section is the same as that of the preceding one, and we make the same simplifying hypotheses.

\begin{lemma}
\label{lemmaM}
Let $d\ge3$ be an integer.  Suppose that, for $j=1,\dots,d-3$, there exists an integer $m_j\neq 0$ and a real number $\kappa_j>0$ satisfying
\begin{equation}
\label{lemmaM:eq1}
 \{ x_{k,0}\cdots x_{k+j-1,0}x_{k+j+1,0}\xi^{j+3} \pm m_j x_{k+1,0}\xi^3 \}
 \le \kappa_j X_{k+1}^{-1}
\end{equation}
for any integer $k\ge 1$, with a choice of sign $\pm$ depending on $j$ and $k$.  Then, there exist constants $c_1,c_2>0$ such that, for any polynomial $R\in\bZ[T]$ of degree $d$ and any integer $k\ge 2$ with $X_{k-1}>c_1\|R\|$, we have
\begin{equation}
\label{lemmaM:eq2}
 \{ x_{k+d-2,0} R(\xi) \}  \ge c_2 X_k X_{k+d-2}^{-1}.
\end{equation}
\end{lemma}

By Lemma \ref{prelim:lemmaX} (i) and (ii), the condition \eqref{lemmaM:eq1} holds true for $j=1$ and $j=2$ with $m_1=2$ and $m_2=6$.  In Section \ref{sec:deg6}, we will prove that it also holds for $j=3$ with $m_3=20$ (see Lemma \ref{lemma:X6} (ii)).  Therefore the hypotheses of the above lemma are satisfied for $d=3,\dots,6$.  Numerical experiments seem to indicate that there are satisfied at least up to $d=9$ (with $m_4=80$, $m_5=360$ and $m_6=1840$).

\begin{proof}
For each $k\ge 1$ and each $j=0,\dots,d-3$, define
\[
 q_{k,j}=x_{k,0} \cdots x_{k+j-1,0}
\]
with the convention that, for $j=0$, this gives $q_{k,0}=1$.  Then the hypothesis means that, for each $k\ge 1$ and $j=1,\dots,d-3$, there exists an integer $n_{k,j}$ and a choice of sign $\pm$ for which
\[
 q_{k,j}x_{k+j+1,0}\xi^{j+3} = \pm m_j x_{k+1,0} \xi^3 + n_{k,j} + \cO(X_{k+1}^{-1}).
\]
This equality also holds for $j=0$ by putting $m_0=1$ and $n_{k,0}=0$.  Upon multiplying both sides of that equality by $x_{k-1,0}$ and using Lemma \ref{lemmaAF} (a), this leads to
\begin{equation}
\label{lemmaM:eq3}
 q_{k-1,j+1}x_{k+j+1,0}\xi^{j+3} = \pm m_j A_{k-1} + n_{k,j}x_{k-1,0} + \cO(X_{k}^{-1}).
\end{equation}
For simplicity, define $q_k:=q_{k-1,d-2}$.  Then, for $j=d-3$, the above equality becomes
\begin{equation}
\label{lemmaM:eq4}
 q_k x_{k+d-2,0}\xi^{d} = \pm m_{d-3} A_{k-1} + n_{k,d-3}x_{k-1,0} + \cO(X_{k}^{-1}).
\end{equation}
For $j=0,\dots,d-4$, we replace $k$ by $k+d-j-3$ in \eqref{lemmaM:eq3} and multiply both sides of the resulting equality by $q_{k-1,d-j-3}$.  As the latter integer is divisible by $x_{k-1,0}$ and as its absolute value is $\asymp X_{k+d-j-3}X_k^{-1}$, this gives
\begin{equation}
\label{lemmaM:eq5}
 q_k x_{k+d-2,0} \xi^{j+3} = p_{k,j+3}x_{k-1,0} + \cO(X_{k}^{-1}) \quad (j=0,\dots,d-4)
\end{equation}
for some integers $p_{k,j+3}$.  We also find
\begin{equation}
\label{lemmaM:eq6}
 q_k x_{k+d-2,0} \xi^{j}
  = q_k x_{k+d-2,j} + \cO(X_{k}^{-1})
  = p_{k,j}x_{k-1,0} + \cO(X_{k}^{-1}) \quad (j=0,1,2)
\end{equation}
where $p_{k,j}=q_{k,d-3}x_{k+d-2,j}\in \bZ$.

\smallskip
Now, let $R$ be an arbitrary polynomial of $\bZ[T]$ of degree $d$ and let $r$ denote its leading coefficient.  Using \eqref{lemmaM:eq4}, \eqref{lemmaM:eq5} and \eqref{lemmaM:eq6}, we obtain that, for each $k\ge 2$, there exists an integer $p_k$ such that
\[
 q_k x_{k+d-2,0} R(\xi)
  = \pm r m_{d-3}A_{k-1} + p_k x_{k-1,0} + \cO(X_{k}^{-1}\|R\|).
\]
Therefore there is a constant $c>0$ such that
\begin{equation}
\label{lemmaM:eq7}
 \{ x_{k+d-2,0} R(\xi) \}
  \ge \Big\{ \frac{r m_{d-3}A_{k-1} \pm p_k x_{k-1,0}}{q_k} \Big\} - \frac{c\,\|R\|}{|q_k| X_{k}}
  \quad (k\ge 2).
\end{equation}
Recall that $x_{k-1,0}$ divides $q_k$.  So, for the ratio $(r m_{d-3}A_{k-1} \pm p_k x_{k-1,0})/q_k$ to be an integer, $x_{k-1,0}$ has to divide $r m_{d-3}A_{k-1}$.  By Proposition \ref{propGCD}, this in turn requires that $x_{k-1,0}$ divides $2r m_{d-3}$.

Define $c_1= 2 \max\{c, |m_{d-3}|(1+\xi^2)\}$ and assume that $X_{k-1}>c_1\|R\|$.  By the choice of $c_1$, this hypothesis together with \eqref{prelim:hyp} leads to
\[
 |2rm_{d-3}| \le 2 \|R\| |m_{d-3}| < (1+\xi^2)^{-1} X_{k-1} \le |x_{k-1,0}|
\]
which implies that $x_{k-1,0}$ does not divide $2rm_{d-3}$.  Then, according to the preceding discussion, the inequality \eqref{lemmaM:eq7} leads to
\begin{equation*}
 \{ x_{k+d-2,0} R(\xi) \}
  \ge \frac{1}{|q_k|}  - \frac{c\,\|R\|}{|q_k| X_{k}}
  \ge \frac{1}{|q_k|}  - \frac{cX_{k-1}}{c_1|q_k| X_{k}}
  \ge \frac{1}{2|q_k|}
  \gg \frac{X_k}{X_{k+d-2}}.
\end{equation*}
\end{proof}

\begin{theorem}
\label{thmR}
Suppose that the hypotheses of Lemma \ref{lemmaM} hold for some integer $d\ge 3$. Then any polynomial $R\in\bZ[T]$ of degree $d$ satisfies $|R(\xi)| \gg \|R\|^{\gamma^2-2\gamma^d}$.
\end{theorem}

In view of the comments following Lemma \ref{lemmaM}, this proves Theorem \ref{intro:thmR} for $d=4,5,6$.

\begin{proof}
Let $R(T)$ be an element of $\bZ[T]$ of degree $d$. Assuming, as we may, that the constant $c_1$ of Lemma \ref{lemmaM} is $\ge X_1$, there exists an integer $k\ge 3$ such that
\[
 X_{k-2} \le c_1 \|R\| < X_{k-1}.
\]
Then, applying Lemma \ref{lemmaM}, we obtain
\[
 X_{k+d-2} |R(\xi)|
  \gg \{ x_{k+d-2,0} R(\xi) \} \gg X_k X_{k+d-2}^{-1},
\]
and so $|R(\xi)| \gg X_{k-2}^{\gamma^2-2\gamma^d} \gg \|R\|^{\gamma^2-2\gamma^d}$.
\end{proof}

\begin{lemma}
\label{lemmaMbis}
Let $d=3$, $4$ or $5$.  There exist positive constants $c_3$ and $c_4$ depending only on $\xi$ with the following property. For any polynomial $R\in\bZ[T]$ of degree $d$ and any integer $k\ge 1$ satisfying $X_{k+d-2} > c_3\|R\|^{\gamma^{d+4}}$, we have
$\{ x_{k+d-2,0}R(\xi) \}  \ge c_4 \|R\|^{\gamma^7-\gamma^{d+5}}$.
\end{lemma}

\begin{proof}
Let $R$ be a polynomial of $\bZ[T]$ of degree $d$.  For each choice of integers $k$ and $\ell$ with $k\ge \ell\ge 1$, $k\equiv \ell$ mod $6$, and $X_{\ell-1}>c_1\|R\|$, Proposition \ref{propdelta} combined with Lemma \ref{lemmaM} gives
\begin{equation}
\label{lemmaMbis:eq1}
\{ x_{k+d-2,0} R(\xi) \}
  = \{ x_{\ell+d-2,0} R(\xi) \} + \cO\Big( \frac{\|R\|}{X_{\ell+d-2}} \Big)
  \ge \frac{c_2 X_\ell}{X_{\ell+d-2}} -  \frac{c\|R\|}{X_{\ell+d-2}}
\end{equation}
for some constant $c>0$ depending only on $\xi$.  Now, choose a constant $c'$ large enough so that it satisfies $c'\ge \max\{X_1,c_1,2c/c_2\}$ and $X_{j+d+4}\le c'X_j^{\gamma^{d+4}}$ for each $j\ge 1$.  Define $c_3=(c')^{1+\gamma^{d+4}}$ and let $k$ be a positive integer such that $X_{k+d-2} > c_3\|R\|^{\gamma^{d+4}}$.  Since $c'\ge X_1$, there exists an integer $m\ge 3$ such that
\[
 X_{m-2} \le c'\|R\| < X_{m-1}.
\]
For this choice of $m$ we have $X_{k+d-2} > c'(c'\|R\|)^{\gamma^{d+4}} \ge c'X_{m-2}^{\gamma^{d+4}} \ge X_{m+d+2}$ and therefore $k\ge m+5$.  Let $\ell$ denote the integer congruent to $k$ modulo $6$ among $\{m,m+1,\dots,m+5\}$.
As $\ell\ge m$ and $c'\ge c_1$, we have $X_{\ell-1}>c_1\|R\|$ and so \eqref{lemmaMbis:eq1} applies.  Since $c \|R\| \le (c/c') X_{m-1} \le c_2 X_\ell/2$, the latter estimate gives
\[
 \{ x_{k+d-2,0}R(\xi) \}
 \ge \frac{c_2 X_\ell}{2X_{\ell+d-2}}
 \asymp X_\ell^{1-\gamma^{d-2}}
 \gg X_{m-2}^{\gamma^7-\gamma^{d+5}}
 \gg \|R\|^{\gamma^7-\gamma^{d+5}}.
\]
\end{proof}

\begin{theorem}
\label{thmPR}
For any polynomial $R\in\bZ[T]$ of degree $d\in\{3,4,5\}$ and any polynomial $P\in\bZ[T]$ of degree at most $2$, we have
$|R(\xi)+P(\xi)| \gg (1+\|P\|)^{-\gamma}\|R\|^{-\gamma^{d+7}}$.
\end{theorem}

\begin{proof}  Fix $P$, $R$ and $d$ as in the statement of the theorem.  Arguing as in the proof of Theorem
\ref{thmPR3}, we first note that, for any $k\ge 1$, we have
\begin{equation}
\label{thmPR:eq1}
 \{ x_{k+d-2,0} R(\xi) \}
 \le X_{k+d-2} |R(\xi)+P(\xi)| + c X_{k+d-2}^{-1} \|P\|
\end{equation}
with a constant $c>0$ depending only on $\xi$.  Define $c'=\max\{X_{d-2},\, c_3,\, 2c/c_4\}$ where $c_3$ and $c_4$ are as in Lemma \ref{lemmaMbis}.  Then, there exists an integer $k\ge 1$ such that
\[
 X_{k+d-3} \le c' (1+\|P\|) \|R\|^{\gamma^{d+5}} < X_{k+d-2}.
\]
For this choice of $k$, we find that $X_{k+d-2} > c_3 \|R\|^{\gamma^{d+5}}$, so Lemma \ref{lemmaMbis} gives $\{ x_{k+d-2,0} R(\xi) \} \ge c_4\|R\|^{-\gamma^{d+5}}$.  We also find that $c X_{k+d-2}^{-1} \|P\| \le (c_4/2)\|R\|^{-\gamma^{d+5}}$.  Substituting these estimates into \eqref{thmPR:eq1}, we obtain
\[
 X_{k+d-2} |R(\xi)+P(\xi)| \ge (c_4/2) \|R\|^{-\gamma^{d+5}}.
\]
The conclusion follows since $X_{k+d-2} \ll (1+\|P\|)^\gamma\|R\|^{\gamma^{d+6}}$.
\end{proof}

%%%%%%%%%%%%%%%%%%%%%%%%%%%%%%%%%%%%%%%%%%%%%%%%%%%%%%%%%%%%%%%%%%
%
%   Degree 6
%
%%%%%%%%%%%%%%%%%%%%%%%%%%%%%%%%%%%%%%%%%%%%%%%%%%%%%%%%%%%%%%%%%%

%\newpage
\section{Degree 6}
\label{sec:deg6}

This last section is devoted to the proof of Theorem \ref{intro:thmdeg6}.  It also provides the needed estimate in the proof of Theorem \ref{intro:thmR} for $d=6$ (see Section \ref{sec:deg4-5}).  We start by establishing general estimates and commutation formulas which complement those of Section \ref{sec:prelim}.

\begin{lemma}
\label{lemma:Bil}
For each integer $k\ge 1$, we have\\[5pt]
$
\begin{array}{rrcl}
\quad \mathrm{(i)}
&x_{k,0}^2\xi
   &= &x_{k,0}x_{k,1} - (-1)^k/3 + \cO(X_k^{-2}),\\[5pt]
\quad \mathrm{(ii)}
&x_{k,0}x_{k,1}\xi
   &= &x_{k,1}^2 - (-1)^k\xi/3 + \cO(X_k^{-2}),\\[5pt]
\quad \mathrm{(iii)}
&x_{k,0}x_{k,2}\xi
   &= &x_{k,1}x_{k,2} - (-1)^k\xi^2/3 + \cO(X_k^{-2}),\\[5pt]
\quad \mathrm{(iv)}
&x_{k,1}^2\xi
   &= &x_{k,1}x_{k,2} - \xi - (-1)^k\xi^2/3 + \cO(X_k^{-2}),\\[5pt]
\quad \mathrm{(v)}
&x_{k,1}x_{k,2}\xi
   &= &x_{k,2}^2 - \xi^2 - (-1)^k\xi^3/3 + \cO(X_k^{-2}).
\end{array}
$
\end{lemma}

\begin{proof} Applying first the basic estimate \eqref{prelim:eq:basics} and then the commutation formula \eqref{prelim:eq:2,0,1}, we obtain
\[
 x_{k,0}x_{k+2,0}\xi
  = x_{k,0}x_{k+2,1} + \cO(X_{k+1}^{-1})
  = x_{k,1}x_{k+2,0} - (-1)^kx_{k+1,0} + \cO(X_{k+1}^{-1})
  \]
which, after multiplication by $x_{k,0}/x_{k+2,0}$, gives
\[
 x_{k,0}^2\xi
  = x_{k,0}x_{k,1} - (-1)^k\frac{x_{k,0}x_{k+1,0}}{x_{k+2,0}} + \cO(X_{k+1}^{-2}).
\]
The estimate (i) follows from this since the recurrence formula \eqref{prelim:eqrecj} with $j=0$ gives
\[
 x_{k,0}x_{k+1,0}
  = (x_{k+2,0} + x_{k-1,0})/3
  = x_{k+2,0}/3 + \cO(X_{k-1}).
\]
To get (ii) and (iii), it suffices to multiply both sides of (i) by $x_{k,j}/x_{k,0}$ for $j=1,2$ and to simplify the resulting expression using the fact that $x_{k,j}/x_{k,0} = \xi^j+\cO(X_k^{-2})$. The estimate (iv) follows from (iii) because, since $\ux_k\in\SL_2(\bZ)$, we have $x_{k,1}^2 = x_{k,0}x_{k,2}-1$.  Finally, (v) derives from (iv) upon multiplying both sides of this estimate by $x_{k,2}/x_{k,1}$ and noting that $x_{k,2}/x_{k,1} = \xi+\cO(X_k^{-2})$.
\end{proof}

\begin{lemma}
\label{lemma:Comm3}
For each integer $k\ge 1$, we have\\[5pt]
$
\begin{array}{rrcl}
\quad \mathrm{(i)}
&x_{k,0}x_{k+3,1}
   &= &x_{k,1}x_{k+3,0} - 3(-1)^k x_{k+1,0}^2,\\[5pt]
\mathrm{(ii)}
&x_{k,0}x_{k+3,2}
   &= &x_{k,2}x_{k+3,0} - 3x_{k+1,0}( 3x_{k+1,0} + 2(-1)^k x_{k+1,1} ),\\[5pt]
\mathrm{(iii)}
&x_{k,1}x_{k+3,2}
   &= &x_{k,2}x_{k+3,1} - 3x_{k+1,0}( 3x_{k+1,1} + (-1)^k x_{k+1,2} ).
\end{array}
$
\end{lemma}

\begin{proof} By Lemma \ref{prelim:lemma:identities} (i), we have $\ux_{k+3}=3x_{k+1,0}\ux_{k+2}-\ux_k$.  Upon multiplying both sides of this equality on the left by $\ux_kJ$ and using Lemma \ref{prelim:lemma:identities} (iii), we obtain
\[
 \ux_k J \ux_{k+3} = 3x_{k+1,0}JM_k\ux_{k+1} - J.
\]
Then (i) and (iii) follow by comparing the diagonal entries of the matrices on both sides of this equality, while (ii) follows by comparing the sum of their off-diagonal entries.
\end{proof}

\begin{lemma}
\label{lemma:Est3}
For each integer $k\ge 1$, we have\\[5pt]
$
\begin{array}{rrcl}
\quad \mathrm{(i)}
&x_{k,0}x_{k+3,2}\xi
   &\equiv &-2\xi + \cO(X_{k+1}^{-2}) \quad \mod \bZ,\\[5pt]
\mathrm{(ii)}
&x_{k,0}x_{k+3,2}\xi^2
   &\equiv &-4\xi^2 + \cO(X_{k+1}^{-2}) \quad \mod \bZ,\\[5pt]
\mathrm{(iii)}
&x_{k,1}x_{k+3,2}\xi
   &\equiv &-3(-1)^k\xi - \xi^2 + \cO(X_{k+1}^{-2}) \quad \mod \bZ.
\end{array}
$
\end{lemma}

\begin{proof}
Applying first Lemma \ref{lemma:Comm3} (ii) and then Lemma \ref{lemma:Bil} (i) and (ii), we find
\begin{align*}
 x_{k,0}x_{k+3,2}\xi
   &= x_{k,2}x_{k+3,0}\xi - 3x_{k+1,0}( 3x_{k+1,0} + 2(-1)^k x_{k+1,1} )\xi,\\
   &= x_{k,2}x_{k+3,1} - 3x_{k+1,1}( 3x_{k+1,0} + 2(-1)^k x_{k+1,1} ) -3(-1)^k -2\xi + \cO(X_{k+1}^{-2}),
\end{align*}
which proves (i).  Then, multiplying this equality by $\xi$ and applying Lemma \ref{lemma:Bil} (ii) and (iv), a short computation gives (ii).  Similarly (iii) follows from Lemma \ref{lemma:Comm3} (iii) together with Lemma \ref{lemma:Bil} (ii) and (iii).
\end{proof}

The first step in the proof of Theorem \ref{intro:thmdeg6} is to show that there are arbitrarily large values of $k$ for which $\{x_{k,0}\xi^6\} \le c/k$ for some constant $c=c(\xi)>0$.  This is achieved by proving first that the sequence of differences $\{(x_{k+6,0}-x_{k,0})\xi^6\}$ admits at most six accumulation points and then that these are irrational numbers.  The reader will note that this already implies that the numbers $\{x_{k,0}\xi^6\}$ are dense in the interval $[0,1/2]$.  We start by working out an estimate for $\{(x_{k+6,0}-x_{k,0})\xi^6\}$ modulo $\bZ$.

\begin{proposition}
 \label{prop:sigma}
For each $k\ge 1$, the quantity $\sigma_k := (x_{k+6,0}-x_{k,0})\xi^6$ satisfies
\[
 \sigma_k
 \equiv -18(-1)^k \big( 2D_{k+1}\xi +12x_{k,0}\xi^3+(-1)^kx_{k+3,0}\xi^4 \big)
   + \cO(X_k^{-1})
 \mod \bZ.
\]
\end{proposition}

\begin{proof}
Going back to the computations of Lemma \ref{lemmaP} (ii) for $j=5$ and replacing at each step the congruence modulo $\bZ$ by an explicit estimate, we find
\begin{equation}
 \label{prop:sigma:eq0}
 \begin{aligned}
 \frac{1}{9}(x_{k+6,0}-x_{k,0})\xi^5
 =\, &x_{k+2,1}B_{k+3} - x_{k+1,2}C_{k+3} + (-1)^kB_{k+2}\\
  &- 9C_{k+1} - 2(-1)^kD_{k+1} - (-1)^kF_{k+1} +\cO(X_k^{-1}).
 \end{aligned}
\end{equation}
Now, we multiply both sides of this equality by $\xi$ and estimate separately each product in the right hand side.  We first find
\begin{equation}
 \label{prop:sigma:eq1}
 \begin{aligned}
 x_{k+2,1}B_{k+3}\xi
  &= x_{k+2,1}x_{k+3,2}x_{k+5,2}\xi - 3x_{k+2,1}x_{k+4,2}\xi\\
  &= \big( x_{k+2,2}x_{k+3,1} - 3x_{k+1,1} -(-1)^kx_{k+1,2} \big) x_{k+5,2}\xi - 3x_{k+2,1}x_{k+4,2}\xi\\
  &\equiv (-1)^k x_{k+2,2}x_{k+4,2}\xi - (-1)^k x_{k+1,2}x_{k+5,2}\xi + \cO(X_{k+3}^{-1}) \mod \bZ,
 \end{aligned}
\end{equation}
where the second equality follows from the commutation formula \eqref{prelim:eq:1,1,2} applied to the product $x_{k+2,1}x_{k+3,2}$ while the congruence modulo $\bZ$ derives from Lemma \ref{lemmaAF} (b) and (f).  Thanks to \eqref{LemmaP:eqCk}, we also have
\begin{equation}
 \label{prop:sigma:eq2}
 C_{k+1}\xi \equiv 2(-1)^k x_{k,2}\xi + \cO(X_k^{-1}) \mod \bZ.
\end{equation}
From this, we deduce that
\begin{equation}
 \label{prop:sigma:eq3}
 x_{k+1,2}C_{k+3}\xi \equiv 2(-1)^k x_{k+1,2}x_{k+2,2}\xi + \cO(X_k^{-1}) \mod \bZ.
\end{equation}
The formula for $F_k$ is not stated in Lemma \ref{lemmaAF} but is easily derived from \eqref{prelim:eq:4,1,2}.  Expanding it for $F_{k+1}$ and then using Lemma \ref{lemmaAF} (a) and (b), we obtain
\begin{equation}
 \label{prop:sigma:eq4}
 \begin{aligned}
 F_{k+1}\xi
  &= x_{k+1,2}x_{k+5,2}\xi - 3(3x_{k+2,0}-(-1)^kx_{k+2,1})x_{k+4,2}\xi \\
  &\equiv x_{k+1,2}x_{k+5,2}\xi - 3 x_{k+3,2}\xi + \cO(X_{k+3}^{-1}) \mod \bZ.
 \end{aligned}
\end{equation}
Combining (\ref{prop:sigma:eq0}--\ref{prop:sigma:eq4}) and using the expressions for $B_{k+2}$ and $D_{k+1}$ coming from their definitions in Lemma \ref{lemmaAF}, we obtain after simplifications
\[
 \frac{(-1)^k\sigma_k}{9}
 \equiv
 2x_{k+2,2}x_{k+4,2}\xi - 2x_{k+1,2}x_{k+5,2}\xi - 4x_{k+1,2}x_{k+2,2}\xi - 12x_{k,2}\xi + \cO(X_k^{-1}) \mod \bZ.
\]
Since by \eqref{LemmaP:eq:xi^3}, we have $x_{k+5,2}\xi \equiv - x_{k+2,2}\xi + \cO(X_{k+2}^{-1}) \mod \bZ$, the above congruence simplifies to
\begin{equation}
 \label{prop:sigma:eq5}
 \begin{aligned}
 (-1)^k\sigma_k/9
 &\equiv
 2x_{k+2,2}x_{k+4,2}\xi - 2x_{k+1,2}x_{k+2,2}\xi - 12x_{k,2}\xi + \cO(X_k^{-1}) \mod \bZ\\
 &\equiv
 2x_{k+2,2}x_{k+4,2}\xi - 2D_{k+1}\xi - 18x_{k,2}\xi + \cO(X_k^{-1}) \mod \bZ.
 \end{aligned}
\end{equation}
On the other hand, upon expanding $x_{k+4,2}$ according to the recurrence formula \eqref{prelim:eqrecj} with $j=2$ and then using Lemma \ref{lemma:Bil} (iii) to estimate $x_{k+2,0}x_{k+2,2}\xi$, we find
\[
 \begin{aligned}
 x_{k+2,2}x_{k+4,2}\xi
 &= x_{k+2,2} ( 3x_{k+2,0}x_{k+3,2} - x_{k+1,2} )\xi\\
 &\equiv -(-1)^k x_{k+3,2}\xi^2 - x_{k+1,2}x_{k+2,2}\xi + \cO(X_k^{-1}) \mod \bZ\\
 &\equiv -D_{k+1}\xi - 3x_{k,2}\xi -(-1)^kx_{k+3,2}\xi^2 + \cO(X_k^{-1}) \mod \bZ.
 \end{aligned}
\]
The conclusion follows by substituting this estimate into \eqref{prop:sigma:eq5} and using the basic estimates \eqref{prelim:eq:basics}.
\end{proof}

\begin{corollary}
 \label{cor:sigma:limit}
Let $\sigma_k$ be defined as in Proposition \ref{prop:sigma}.  Then, for each integer $k\ge 1$, the limit $\delta_k := \lim_{i\to\infty}\{\sigma_{k+6i}\}$ exists and is a $6$-periodic function of $k$ satisfying
\[
 \sigma_k \equiv \delta_k +\cO(X_k^{-1}) \mod \bZ.
\]
\end{corollary}

\begin{proof}
It suffices to show that $\sigma_{k+6} - \sigma_k \equiv \cO(X_k^{-1}) \mod \bZ$. In view of the congruence for $\sigma_k$ given by Proposition \ref{prop:sigma}, this follows from the estimates of Lemma \ref{lemmaP} (ii) for $j=3,4$ together with the fact that \eqref{LemmaP:eq:Dxi} leads to
\[
 \begin{aligned}
 (D_{k+7} - D_{k+1})\xi
 &= (D_{k+7} - D_{k+4})\xi + (D_{k+4} - D_{k+1})\xi \\
 &\equiv -6x_{k+6,2}\xi + 6 x_{k,2}\xi + \cO(X_k^{-1}) \mod \bZ\\
 &\equiv -6(x_{k+6,0} - x_{k,0})\xi^3 + \cO(X_k^{-1}) \mod \bZ\\
 &\equiv \cO(X_k^{-1}) \mod \bZ.
 \end{aligned}
\]
\end{proof}

In order to show that the limit points $\delta_k$ are irrational, we use the following estimates.

\begin{proposition}
 \label{prop:deg6:den_sigma}
For each integer $k\ge 4$, the real number $\sigma_k$ defined in Proposition \ref{prop:sigma} satisfies
\begin{itemize}
\item[(i)] $x_{k-2,0}\sigma_k \equiv -36(-1)^kx_{k-1,2}\xi + \cO(X_{k-1}^{-1}) \mod \bZ$,
\smallskip
\item[(ii)] $x_{k-3,0}\sigma_k \equiv \cO(X_{k-2}^{-2}) \mod \bZ$.
\end{itemize}
\end{proposition}

\begin{proof}  Thanks to \eqref{LemmaP:eq4}, we first note that
\[
 \begin{aligned}
 x_{k+3,0}\xi^4
  &= 3D_{k+1} + 6(-1)^kx_{k,2}\xi - x_{k,0}\xi^4 + \cO(X_k^{-1})\\
  &\equiv 6(-1)^kx_{k,2}\xi - x_{k,2}\xi^2 + \cO(X_k^{-1}) \mod \bZ.
  \end{aligned}
\]
Substituting this into the congruence for $\sigma_k$ given by Proposition \ref{prop:sigma}, and replacing $D_{k+1}$ by its defining formula from Lemma \ref{lemmaAF} (d), we obtain
\begin{equation}
 \label{lemma:den_sigma:eq1}
 \sigma_k
 \equiv
 -18(-1)^k \big( 2x_{k+1,2}x_{k+2,2}\xi + 12x_{k,2}\xi - (-1)^kx_{k,2}\xi^2 \big) + \cO(X_k^{-1}).
\end{equation}

To study $x_{k-2,0}\sigma_k$, we multiply both sides of this congruence by $x_{k-2,0}$.  Thanks to Lemma \ref{lemmaAF} (a) and (b), this gives
\begin{equation}
 \label{lemma:den_sigma:eq2}
 \begin{aligned}
 x_{k-2,0}\sigma_k
 &\equiv -18(-1)^k \big( 2x_{k-2,0}x_{k+1,2}x_{k+2,2}\xi
         - (-1)^kA_{k-2}\xi \big) + \cO(X_{k-1}^{-1}) \mod \bZ.\\
 &\equiv -36(-1)^k \big( x_{k-2,0}x_{k+1,2}x_{k+2,2}\xi
         + x_{k-1,2}\xi \big) + \cO(X_{k-1}^{-1}) \mod \bZ.
 \end{aligned}
\end{equation}
Now, we expand $x_{k-2,0}x_{k+2,2}$ according to the commutation formula \eqref{prelim:eq:4,0,2}, and then apply Lemma \ref{lemmaAF} (d) and Lemma \ref{lemma:Bil} (v) in this order to get
\[
 \begin{aligned}
 x_{k-2,0}x_{k+1,2}x_{k+2,2}\xi
 &= x_{k+1,2} \big( x_{k-2,1}x_{k+2,1}\xi - 3(-1)^kx_{k-1,0}x_{k+1,1}\xi - (-1)^kx_{k,1}\xi \big)\\
 &\equiv - 3(-1)^kx_{k-1,0}x_{k+1,1}x_{k+1,2}\xi + x_{k-1,2}\xi + \cO(X_{k-1}^{-1}) \mod \bZ\\
 &\equiv - 3(-1)^kx_{k-1,0}(-\xi^2+(-1)^k\xi^3/3) + x_{k-1,2}\xi + \cO(X_{k-1}^{-1}) \mod \bZ\\
 &\equiv \cO(X_{k-1}^{-1}) \mod \bZ.
 \end{aligned}
\]
The congruence (i) follows by substituting this expression into \eqref{lemma:den_sigma:eq2}.

The proof of (ii) is similar.  We first multiply both sides of \eqref{lemma:den_sigma:eq1} by $x_{k-3,0}$.  Using Lemma \ref{lemma:Est3} (i) and (ii), this gives
\begin{equation}
 \label{lemma:den_sigma:eq3}
 x_{k-3,0}\sigma_k
 \equiv
 -36(-1)^k \big( x_{k-3,0}x_{k+1,2}x_{k+2,2}\xi - 12\xi + 2(-1)^k\xi^2 \big) + \cO(X_{k-2}^{-2}).
\end{equation}
Expanding $x_{k-3,0}x_{k+1,2}$ according to the commutation formula \eqref{prelim:eq:4,0,2} and then using Lemma \ref{lemmaAF} (b) and (d), we also find
\[
 \begin{aligned}
 x_{k-3,0}x_{k+1,2}x_{k+2,2}\xi
  \equiv (-1)^k x_{k-3,1}x_{k,2}\xi &- 3x_{k-2,0}x_{k+1,2}\xi\\
   &+ (-1)^k x_{k-1,1}x_{k+2,2}\xi + \cO(X_{k-2}^{-2}) \mod \bZ.
 \end{aligned}
\]
The congruence (ii) then follows by substituting this expression into \eqref{lemma:den_sigma:eq3} and then by simplifying the result using Lemma \ref{lemma:Est3} (i) and (iii).
\end{proof}

\begin{theorem}
 \label{thm:deg6:trans}
For each $\ell=1,\dots,6$, the number $\delta_\ell$ defined in Corollary \ref{cor:sigma:limit} is transcendental over $\bQ$ and satisfies $|\delta_\ell-\alpha| \ll H(\alpha)^{-\gamma^3}$ for infinitely many $\alpha\in\bQ$.
\end{theorem}

\begin{proof}
For each $k\ge 4$, denote by $s_k$ and $s_k'$ the integers which are respectively closest to $x_{k-3,0}\sigma_k$ and to $x_{k-2,0}\sigma_k$.  By Proposition \ref{prop:deg6:den_sigma}, we have
\begin{equation}
 \label{thm:deg6:trans:eq1}
 x_{k-3,0}\sigma_k = s_k +\cO(X_{k-2}^{-2})
 \et
 x_{k-2,0}\sigma_k = s_k' -36(-1)^kx_{k-1,2}\xi + \cO(X_{k-1}^{-1})
\end{equation}
Eliminating $\sigma_k$ between these two equalities and then applying Lemma \ref{lemmaAF} (a), we get
\[
 \begin{aligned}
 x_{k-2,0}s_k
  &= x_{k-3,0}s_k' -36(-1)^kx_{k-3,0}x_{k-1,2}\xi + \cO(X_{k-2}^{-1})\\
  &= x_{k-3,0}s_k' -36(-1)^kA_{k-3} + \cO(X_{k-2}^{-1}),
 \end{aligned}
\]
which means that
\[
 x_{k-2,0}s_k = x_{k-3,0}s_k' - 36(-1)^kA_{k-3}
\]
for each sufficiently large index $k$.  From this we deduce that $\gcd( x_{k-3,0}, s_k)$ is a divisor of $36\gcd( x_{k-3,0}, A_{k-3})$ which itself divides $72$ according to Proposition \ref{propGCD}.

Now, define $u_k$ to be the closest integer to $\delta_k-\sigma_k$ and put
\[
 \alpha_k = \frac{s_k}{x_{k-3,0}} + u_k.
\]
By the above considerations, $\alpha_k$ is a rational number with denominator $\den(\alpha_k) \asymp X_{k-3}$, and Corollary \ref{cor:sigma:limit} gives
\[
 \delta_k  = \sigma_k + u_k + \cO(X_k^{-1}).
\]
Combining the latter estimate with the first equality in \eqref{thm:deg6:trans:eq1}, we obtain
\[
 |\delta_k-\alpha_k|
  \le \Big| \sigma_k-\frac{s_k}{x_{k-3,0}} \Big| + \cO(X_k^{-1})
  \ll X_k^{-1}.
\]
Since $\delta_k$ depends only on the residue class of $k$ modulo $6$, this shows in particular that the sequence $(\alpha_k)_{k\ge 4}$ is bounded, and so $H(\alpha_k)\asymp \den(\alpha_k)\asymp X_{k-3}$.  Thus, for each pair of positive integers $k$ and $\ell$ with $1\le \ell\le 6$ and $k\equiv \ell \mod 6$, we have $|\delta_\ell-\alpha_k|\ll H(\alpha_k)^{-\gamma^3}$.  By Roth's theorem, this implies that each $\delta_\ell$ is  transcendental over $\bQ$.
\end{proof}

\begin{proposition}
 \label{prop:deg6:pfrac}
There exists a constant $c>0$, not depending on $\xi$, such that the inequality
\[
 \{ x_{k,0}\xi^6 \} \le \frac{c}{k}
\]
holds for infinitely many values of $k\ge 1$.
\end{proposition}

\begin{proof}
By Theorem \ref{thm:deg6:trans}, we have $\delta_1\notin\bQ$.  Therefore, there exist arbitrarily large real numbers $X$ such that the convex body $\cC_X$ of $\bR^2$ defined by
\[
 |x| \le X, \quad |y+x\delta_1| \le (2X+1)^{-1}
\]
does not contain any non-zero point $(x,y)$ of $\bZ^2$.  Fix such a value of $X$, with $X\ge 6$.  Then, by definition,  the first minimum of $\cC_X$ with respect to $\bZ^2$ is at least $1$ and so, according to Minkowski's second convex body theorem, its second minimum is at most $4/\vol(\cC_X)=2+1/X\le 3$.  By a result of Jarn\'{\i}k \cite{J}, we conclude that $\bR^2=\bZ^2+3\cC_X$ (see also \cite[Ch.~2, \S13.2, Thm.~1]{GL}).  In particular, this means that, for any $r\in\bR$, there exists a point $(m,n)\in\bZ^2$ such that $(m,r+n)\in 3\cC_X$, a condition which translates into
\[
 |m| \le 3X \et |r+n+m\delta_1| \le 3/(2X+1) \le 3/(2X).
\]
Applying this result with $r$ replaced by $r+[3X]\delta_1$ and putting $i=[3X]+m$ for a corresponding choice of $m$, we conclude that, for each $r\in\bR$, there exists an integer $i$ satisfying
\[
 0\le i \le 6X \et \{r+i\delta_1\} \le 3/(2X).
\]

Now, choose an integer $\ell$ with $X\le \ell \le 2X$ and $\ell\equiv 1\mod 6$.  By the above, there exists $i\in\bZ$ with $0\le i \le 6X$ such that
\[
 \{ x_{\ell,0}\xi^6 + i\delta_1 \} \le 3/(2X).
\]
Put $k=\ell+6i$.  Since
\[
 x_{k,0}\xi^6
 = x_{\ell,0}\xi^6 + \sum_{j=0}^{i-1} \sigma_{\ell+6j}
 = x_{\ell,0}\xi^6 + i\delta_1 + \sum_{j=0}^{i-1} (\sigma_{\ell+6j}-\delta_{\ell+6j}),
\]
it follows from Corollary \ref{cor:sigma:limit} that
\[
 \{ x_{k,0}\xi^6 \}
 \le \{ x_{\ell,0}\xi^6 + i\delta_1\} + \sum_{j=0}^{i-1} \{\sigma_{\ell+6j}-\delta_{\ell+6j}\}
 \le \frac{3}{2X} + \sum_{j=0}^{i-1} \frac{c_1}{X_{\ell+6j}}
\]
with a constant $c_1>0$ depending only on $\xi$.  If $X$ is large enough, this gives
$\{ x_{k,0}\xi^6 \} \le 2/X$, and, as $k = \ell + 6i \le 38X$, we conclude that $\{ x_{k,0}\xi^6 \} \le 76/k$.
\end{proof}

The next step in the proof of Theorem \ref{intro:thmdeg6} is to establish a lower bound for the numbers $\{x_{k,0}\xi^6\}$.  To this end, we first show that the hypotheses of Lemma \ref{lemmaM} are satisfied for $d=6$ with $m_3=20$.  This is accomplished by part (ii) of the next lemma. Thanks to Theorem \ref{thmR}, this also completes the proof of Theorem \ref{intro:thmR} for $d=6$.

\begin{lemma}
 \label{lemma:X6}
For each integer $k\ge 1$, we have
\begin{itemize}
 \item[(i)]
 $x_{k,0}x_{k+1,0}x_{k+2,2}x_{k+4,2}\xi^2 \equiv 8(-1)^kx_{k+1,2}\xi +\cO(X_{k+1}^{-1}) \mod \bZ$,
 \smallskip
 \item[(ii)]
 $x_{k,0}x_{k+1,0}x_{k+2,0}x_{k+4,0}\xi^6 \equiv 20(-1)^kx_{k+1,2}\xi + \cO(X_{k+1}^{-1}) \mod \bZ$.
\end{itemize}
\end{lemma}

\begin{proof}
Part (i) follows from the following congruences modulo $\bZ$:
\[
 \begin{aligned}
 x_{k,0}x_{k+1,0}x_{k+2,2}x_{k+4,2}\xi^2
  &\equiv x_{k,0}x_{k+2,2}(-4\xi^2+\cO(X_{k+2}^{-2}))
     &&\text{by Lemma \ref{lemma:Est3} (ii),}\\
  &\equiv -4x_{k,0}x_{k+2,0}\xi^4 + \cO(X_{k+1}^{-1}) \\
  &\equiv 8(-1)^kx_{k+1,2}\xi + \cO(X_{k+1}^{-1})
     &&\text{by Lemma \ref{prelim:lemmaX} (i).}
 \end{aligned}
\]

To prove (ii), we first note that Lemma \ref{prelim:lemmaX} (i) provides
\[
 x_{k,0}x_{k+2,0}\xi^4
 = x_{k,2}x_{k+2,2} - 3x_{k+1,0}\xi^2 - 2(-1)^kx_{k+1,0}\xi^3 + \cO(X_{k+1}^{-1}).
\]
In this estimate, we replace the index $k$ by $k+2$ and then multiply both sides of the resulting equality by $x_{k,0}x_{k+1,0}\xi^2$.  The conclusion then follows by estimating the first product in the right hand side using Part (i) above, the third product using Lemma \ref{prelim:lemmaX} (ii), and the second product using the congruence
\[
 x_{k,0}x_{k+1,0}x_{k+3,0}\xi^4
 \equiv 2(-1)^k x_{k,0}x_{k+2,2}\xi + \cO(X_{k+1}^{-1})
 \equiv \cO(X_{k+1}^{-1}) \mod \bZ
\]
which is obtained by applying first Lemma \ref{prelim:lemmaX} (i) and then Lemma \ref{lemmaAF} (a).
\end{proof}

It is possible that the inequality of Proposition \ref{prop:deg6:pfrac} is optimal and that we have $\{x_{k,0}\xi^6\}\gg k^{-1}$ for all $k\ge 1$.  Here, by combining the congruence of Lemma \ref{lemma:X6} (ii) with Corollary \ref{cor:sigma:limit}, we simply prove the following lower bound.

\begin{proposition}
 \label{prop:deg6:lowerbound}
For any integer $k\ge 1$, we have $\{x_{k,0}\xi^6\} \gg k^{-2\gamma^7}$.
\end{proposition}

\begin{proof}
Let $k\ge 5$ and $i\ge 0$ be integers.  By Corollary \ref{cor:sigma:limit}, we have
\[
 x_{k+6i,0}\xi^6
 = x_{k,0}\xi^6 + \sum_{j=0}^{i-1} \sigma_{k+6j}
 \equiv x_{k,0}\xi^6 + i\sigma_{k} +\cO(iX_k^{-1}) \mod\bZ,
\]
where the constant involved in the symbol $\cO$ depends only on $\xi$.  Multiplying both sides of this relation by the integer $x_{k-4,0}x_{k-3,0}x_{k-2,0}$ and then applying Lemma \ref{lemma:X6} (ii) and Proposition \ref{prop:deg6:den_sigma} (ii) to estimate the products in the right hand side, we obtain
\[
 \begin{aligned}
 x_{k-4,0}x_{k-3,0}x_{k-2,0}x_{k+6i,0}\xi^6
 &\equiv 20(-1)^k x_{k-3,2}\xi + \cO(iX_{k-3}^{-1}) \\
 &\equiv 20(-1)^k x_{k-3,0}\xi^3 + \cO(iX_{k-3}^{-1})\mod\bZ.
 \end{aligned}
\]
By Lemma \ref{lemmaxRbis} applied with $R=20T^3$, the quantity $\{20x_{k-3,0}\xi^3\}$ is bounded below by a positive constant $c_1$ depending only on $\xi$.  Thus, the above congruence implies the existence of a constant $c_2 = c_2(\xi)>0$ such that
\[
 \{x_{k-4,0}x_{k-3,0}x_{k-2,0}x_{k+6i,0}\xi^6\} \ge c_1/2
 \quad\text{whenever}\quad 0\le i\le c_2X_{k-3},
\]
and so $\{x_{k+6i,0}\xi^6\} \gg X_{k-2}^{-2}$ for the same values of $i$.

The conclusion follows easily by observing that, for any sufficiently large integer $j$, there exists an integer $k\ge 10$ with $k\equiv j \mod 6$ such that $c_2X_{k-9} < j \le c_2X_{k-3}$.   Assuming $k\le c_2X_{k-9}$ as we may, this integer takes the form $j=k+6i$ for some $i\in\bZ$ with $0\le i\le c_2X_{k-3}$ and so $\{x_{j,0}\xi^6\} \gg X_{k-2}^{-2} \gg j^{-2\gamma^7}$.
\end{proof}

We conclude by proving Theorem \ref{intro:thmdeg6} in the following more precise form.

\begin{proposition}
 \label{prop:Pk}
For each integer $k\ge 1$, denote by $s_k$ the closest integer to $x_{k,0}\xi^6$ and, when $k\ge 2$, define
\[
 P_k(T) = 2T^6 +(-1)^k( s_{k-1}Q_k(T) - s_kQ_{k+1}(T) + s_{k+1}Q_{k-1}(T) ).
\]
Then, there is a constant $c=c(\xi)>0$ such that, for infinitely many values of $k$, the polynomial $P_k$ admits a root $\alpha_k$ with $H(\alpha_k)$ arbitrarily large and
\begin{equation}
 \label{prop:Pk:eq1}
 |\xi-\alpha_k| \le c H(\alpha_k)^{-\gamma-1}(\log \log H(\alpha_k))^{-1}.
\end{equation}
\end{proposition}

\begin{proof}
Thanks to the formula \eqref{lemmaQ:eq2} of Lemma \ref{lemmaQ}, we have
\[
 \begin{aligned}
 P_k(T)
 = 2(T^6-\xi^6)
   + (-1)^k\big(
    (s_{k-1}-x_{k-1,0}\xi^6)Q_k(T) &- (s_{k}-x_{k,0}\xi^6)Q_{k+1}(T)\\
     &+ (s_{k+1}-x_{k+1,0}\xi^6)Q_{k-1}(T) \big).
 \end{aligned}
\]
Combining this with the estimates \eqref{lemmaQ:eq1} of Lemma \ref{lemmaQ}, we deduce that
\[
 \begin{aligned}
 \|P_k\| &\le c_1 \{x_{k,0}\xi^6\} X_k + c_2 X_{k-1},\\
 |P_k'(\xi)| &\ge c_1 \{x_{k,0}\xi^6\} X_k - c_2 X_{k-1},\\
 |P_k(\xi)| &\le c_1 \{x_{k+1,0}\xi^6\} X_{k+1}^{-1} + c_2 X_{k+2}^{-1}
 \end{aligned}
\]
for some positive constants $c_1$ and $c_2$ that are independent of $k$. In view of Proposition \ref{prop:deg6:lowerbound}, we conclude that
\[
 \|P_k\| \ll \{x_{k,0}\xi^6\} X_k,\quad
 |P_k'(\xi)| \gg \{x_{k,0}\xi^6\} X_k \et
 |P_k(\xi)| \ll \{x_{k+1,0}\xi^6\} X_{k+1}^{-1}.
\]
Therefore the root $\alpha_k$ of $P_k$ which is closest to $\xi$ satisfies
\[
 |\xi-\alpha_k|
 \le 6\frac{|P_k(\xi)|}{|P_k'(\xi)|}
 \ll \frac{\{x_{k+1,0}\xi^6\} X_{k+1}^{-1}}{\{x_{k,0}\xi^6\} X_k}
 \ll \{x_{k+1,0}\xi^6\} \|P_k\|^{-\gamma^2}.
\]
By Proposition \ref{prop:deg6:pfrac}, this means that, for infinitely many values of $k$, we have
\[
 |\xi-\alpha_k|
 \le c_3 k^{-1} \|P_k\|^{-\gamma-1}
\]
with a constant $c_3=c_3(\xi)>0$.  For these values of $k$, the estimate \eqref{prop:Pk:eq1} follows because, since $\alpha_k$ is a root of $P_k$, we have on the one hand $\|P_k\| \gg H(\alpha_k)$, and, since $\|P_k\| \ll X_k \ll c_4^{\gamma^k}$ for some constant $c_4>1$, we have on the other hand $k\gg \log\log H(\alpha_k)$. Finally, $H(\alpha_k)$ goes to infinity with $k$ because $\lim_{k\to\infty}|\xi-\alpha_k|=0$ and $\xi$ is transcendental over $\bQ$.
\end{proof}

%%%%%%%%%%%%%%%%%%%%%%%%%%%%%%%%%%%%%%%%%%%%%%%%%%%%%%%%%%%%%%%%%%%%%%
%%%%%%%%%%%%%%%%%%%%%%%%%%%%%%%%%%%%%%%%%%%%%%%%%%%%%%%%%%%%%%%%%%%%%%
%%\newpage

\end{document}